\documentclass[11pt]{article}
\usepackage{hyperref}
\usepackage{amsfonts,mathrsfs,bbm,rawfonts,amsmath,amssymb,amsthm}
\usepackage{fullpage, setspace}

 \newtheorem{thm}{Theorem}[section]
 \newtheorem{lem}[thm]{Lemma}
 \newtheorem{prop}[thm]{Proposition}

 \newtheorem*{defn}{Definition}
 \newtheorem{rem}[thm]{Remark}
 \numberwithin{equation}{section}

 \newcommand{\al}{\alpha}

 \newcommand{\De}{\Delta}
 \newcommand{\ep}{\epsilon}
 \newcommand{\Si}{\Sigma}
 
 \newcommand{\om}{\omega}
 \newcommand{\Om}{\Omega}
 
 \newcommand{\Ga}{\Gamma}

 \newcommand{\ymh}{Yang-Mills-Higgs }

 \newcommand{\F}{\mathcal{F}}
 \newcommand{\E}{\mathcal{E}}
 \renewcommand{\L}{\mathcal{L}}
 
 \newcommand{\A}{\mathscr{A}}
 \renewcommand{\S}{\mathscr{S}}
 \newcommand{\g}{\mathfrak{g}}
 \newcommand{\G}{\mathscr{G}}
 \newcommand{\D}{\mathbb{D}}

 \newcommand{\Real}{\mathbb{R}}
 \newcommand{\Integer}{\mathbb{Z}}

 \newcommand{\abs}[1]{\vert#1\vert}
 \newcommand{\norm}[1]{\Vert#1\Vert}
 
 \newcommand{\ti}{\tilde}
 \newcommand{\p}{\partial}
 \def\<{\langle} \def\>{\rangle}

 \newcommand{\dt}[1]{\frac {d{#1}} {dt}}
 \newcommand{\nn}{\nonumber}
 \newcommand{\na}{\nabla_A}

\begin{document}

\title{Heat flow of Yang-Mills-Higgs functionals {in dimension two}}

\author{Chong Song\footnote{School of Mathematical Science,
Xiamen University, Xiamen 361005, P. R. China},
\  Changyou Wang\footnote{Department of Mathematics, Purdue University, West Lafayette, IN 47907, USA}}

\maketitle

\begin{abstract}
We consider the heat flow of Yang-Mills-Higgs functional where the base manifold
is a Riemannian surface and the fiber is a compact symplectic manifold. We show that the corresponding Cauchy problem  admits a global weak solution for any $H^1$-initial data. Moreover, the solution is smooth except finitely many singularities.
We prove an energy identity at finite time singularities and give a description of the asymptotic behavior at time infinity.
\end{abstract}


\section{Introduction}

Suppose $(\Sigma, g)$ is a compact Riemann surface {without boundary}, $G$ is a compact Lie group {equipped} with a metric, $\g$ is the Lie algebra
of $G$ and $\g^\ast$ is the dual of $\g$, and $P$ is a principal $G$-bundle on $\Sigma$. Let $(M, \om)$ be a compact symplectic manifold which supports a Hamiltonian action of $G$ with moment map
$\mu : M \to \g^*$, and $\pi : \mathcal{F} = P\times_{G} M \to \Sigma$ be the associated {fiber} bundle with {fiber} $M$.
Then $G$ extends to an equivariant action on $\mathcal F$, and $\mu$ extends to a map on the bundle $\mu:\F\to P\times_{ad}\mathfrak{g}^\ast$.
Denote the space of smooth connections on $P$ by $\A$, the space of smooth sections on $\F$ by $\S$. Denote the $W^{k,p}$-Sobolev completions of the spaces $\A$ and $\S$ by $\A_{k,p}$ and $\S_{k,p}$ respectively.
Then for a pair $(A,\phi) \in \A_{1,2}\times \S_{1,2}$, the \emph{\ymh functional} is defined by
\begin{equation}\label{YMH1}
\E(A,\phi) := \norm{F_A}^2_{L^2(\Sigma)}+ \norm{D_A \phi}^2_{L^2(\Sigma)} + \norm{\mu(\phi) - c}^2_{L^2(\Sigma)},
\end{equation}
where $F_A$ is the curvature of $A$, {$D_A$ is the exterior derivative associated with connection $A$},
 and $c \in \mathfrak{g}^\ast$ is a fixed central element.

The \emph{Yang-Mills-Higgs (or YMH) functional} is a composition of the famous \emph{Yang-Mills functional}, the kinetic energy functional and the Higgs potential {energy functional}. Since the YMH functional appears naturally in
{the} classical gauge theory, it {has generated a lot of} interests among {both} physicists and mathematicians during the past decades. For example, the Ginzburg-Landau equation in the superconductivity theory coincides with the variational equation of YMH functional. The critical points of the YMH functional are known as \emph{\ymh fields} and the YMH functional is an appropriate Morse function to study the underlying spaces (\cite{AB, JT, P, So, T}). On the other hand, it is also well-known that the minimal YMH fields are the so-called \emph{symplectic vortices} and their moduli space can be used to define invariants on symplectic manifolds with Hamiltonian actions \cite{CGS,M}.

A natural method to study the existence of critical points is the heat flow method. If we denote the formal adjoint operator{s} of the exterior derivative $D_A$ and {the} covariant derivative $\nabla_A$ by $D_A^*$ and $\nabla_A^*$ respectively, then the
equation of heat flow of \ymh functional can be written as
\begin{equation}\label{e:heat}
\left\{
\begin{aligned}
\frac{\partial A}{\partial t} &= -D_A^*F_A - \phi^* D_A\phi, \\
\frac{\partial\phi}{\partial t} &= -D_A^*D_A\phi - (\mu(\phi) - c)\cdot \nabla\mu(\phi).
\end{aligned}
\right.
\end{equation}
Here $\nabla$ denotes the connection induced by the metric on $(M,\omega)$ and $\phi^* D_A\phi$
stands for {the} element in the dual space of $\Om^1(AdP)$ consisting of $AdP$-valued 1-forms,
with $AdP=P\times_{Ad}\g$, {which is defined as follows:}
\[ \int_\Si \langle\phi^* D_A\phi, B\rangle dv_g := \int_\Si \langle D_A\phi, B\phi \rangle dv_g,
\ \forall\ B\in \Om^1(AdP). \]

Here we would like to mention a few relevant references from the vast literatures concerning both the YMH flow and the closely related Yang-Mills flow. The heat flow of Yang-Mills functional was first suggested by Atiyah and Bott~\cite{AB} and has been studied by many people. For example, R{\aa}de~\cite{R} studied the Yang-Mills heat flow in dimensions 2 and 3, and Struwe~\cite{St} has studied the Yang-Mills flow in dimension 4 and showed both the existence and uniqueness of local smooth solutions, while Schlatter~\cite{Sc1, Sc2} has described the blow-up phenomenon and long time behavior of the Yang-Mills flow. Recently, Hong, Tian and Yin \cite{HTY} have applied the Yang-Mills $\alpha$-flow to construct a weak solution of the Yang-Mills flow in dimension 4. For the YMH flow of a vector bundle over 4 dimensional manifold, results similar to \cite{St, Sc1} were obtained by Fang and Hong \cite{FH}. Hong and Tian \cite{HT} have also studied the asymptotic behavior of both the Yang-Mills flow and the YMH flow in higher dimensions. It is worth mentioning that  the Yang-Mills flow and the YMH flow have also been employed in the study of the existence of Hermitian-Einstein metrics and the so-called Hitchin-Kobayashi correspondence(cf. \cite{D, UY, Si, H, LZ}). For more related results, we refer to the book by Feehan~\cite{F} and references therein.

The equation of heat flow of YMH functional~(\ref{e:heat}) has its independent interest from the analytic point of view. It is a coupled quasilinear degenerate parabolic system with critical nonlinearities: the first equation in~(\ref{e:heat}) is analogous {to} the  Yang-Mills flow, while the second equation {in~(\ref{e:heat})} can be viewed
as the gauged heat flow of harmonic maps. Thus  it is {natural to expect} that
the equation of YMH flow \eqref{e:heat} should reflect some  features that are common with respect to
both the {Yang-Mills flow} and the heat flow of harmonic maps. It is worthwhile to mention that
(i) since the dimension of the base manifold ($\Sigma, g$) is two,  the nonlinearity of first equation of YMH flow \eqref{e:heat} becomes subcritical, while the nonlinearity of second equation of YMH  flow \eqref{e:heat} is critical;
(ii) since the fiber is a compact manifold that may support nontrivial harmonic maps from $\mathbb S^2$, the second
equation of  YMH  flow \eqref{e:heat} may develop finite time singularity.
Recall that the critical dimension of {Yang-Mills
 flow} is {four} in which the {Yang-Mills} functional is conformally invariant.
R{\aa}de~\cite{R} {has} showed that the {Yang-Mills  flow} admits a global smooth solution in subcritical dimensions 2 and 3, while the existence of global smooth solutions to the Yang-Mills flow in dimension 4 is an outstanding open problem. Yu~\cite{Y} showed the local existence of the YMH flow in dimension 2 with smooth initial data and studied the bubbling analysis at the first singular time.
In this paper, we will show that a weaker version of R{\aa}de's result holds for the YMH  flow
\eqref{e:heat}, namely there exists a global weak solution of (\ref{e:heat}) that is smooth away from finitely many points.
Note that finite time singularities do occur for harmonic map heat flows in dimension two~\cite{CDY,CL} (see also \cite{YL} for other related work), which could be regarded as a special case of the YMH flow.

Before stating our results, we first give the definition of weak solutions.
\begin{defn} For $0<T\le +\infty$ and $(A_0,\phi_0)\in \A_{0,2}\times \S_{1,2}$,
a pair of sections $(A,\phi)$
is called a weak solution to the YMH flow equation~(\ref{e:heat}) under the initial condition
\begin{equation} \label{IVP}
(A,\phi)\big|_{t=0}=(A_0,\phi_0),
\end{equation}
on the interval $[0, T)$, if \begin{itemize}
\item[(i)] $(A, \phi)\in C^0\big([0, T), \A_{0,2}\big)\times \Big(C^0([0,T], \S_{0,2})\cap L^2([0,T), \S_{1,2})\Big)$,
and $F_A\in L^2([0,T), L^2)$.
\item [(ii)] the following holds:
\begin{equation}\label{e:weak-sol}
\left\{
\begin{aligned}
&\int_0^T\int_\Si\< A, \p_t B\> \,dv_gdt = \int_0^T\int_\Si\big(\< F_A, D_AB\> + \< D_A\phi, B\phi\>\big)\,dv_gdt;\\
&\int_0^T\int_\Si\< \phi, \p_t \psi\> \,dv_gdt = \int_0^T\int_\Si\big(\< D_A\phi, D_A\psi\> + \< (\mu(\phi)-c)\nabla\mu(\phi), \psi\>\big) \,dv_gdt,\\
\end{aligned}
\right.
\end{equation}
for any test functions $B\in H_0^1\big([0, T], \A_{0,2}\big)\cap L^2\big([0,T), \A_{1,2}\big)$,
and $\psi\in H^1_0\big([0,T), \S_{0,2})\big)\cap L^2\big([0,T], \S_{1,2}\big)$.
\item [(iii)] $(A, \phi)$ satisfies
(\ref{IVP}) in $L^2$-sense.
\end{itemize}
\end{defn}

Now we state our main theorem on the existence of global weak solutions to the YMH flow.
\begin{thm}\label{main}
Let $(A_0, \phi_0)\in \A_{1,2}\times \S_{1,2}$. There exist a global weak solution $(A, \phi)$ to the YMH  flow~(\ref{e:heat})
and (\ref{IVP}) such that
\begin{itemize}
\item[i)] the energy inequality $\E(A(t),\phi(t))\le\E(A_0,\phi_0)$ holds for all $0\le t<+\infty$, and
\[ A\in C^0([0, \infty), \A_{0,2});\ \phi\in C^0([0, \infty), \S_{0,2}); \ F_A\in L^\infty([0, \infty), L^2);
\ D_A\phi\in L^\infty([0,\infty), L^2).\]
\item[ii)] there exist a positive integer $L\le [\frac{\E(0)}{\alpha(M)}]$, and gauge transformations
$\{s_i\}_{i=1}^L\subset\G_{2,2}$, and $0=T_0<T_1<T_2<\cdots<T_L<+\infty$ such that for $1\le i\le L$,
$(s_i^*A,s_i^*\phi)\in C^\infty\big(\Sigma\times (T_{i-1}, T_i]\setminus \{(x_1^i,T_i),\cdots, (x_{j(i)}^i,T_i)\}\big)$ for some $j(i)\le
\frac{\E(0)}{\alpha(M)}$. Here $\alpha(M)>0$ is defined by
$$\alpha(M)=\inf\big\{\int_{\mathbb S^2}|\nabla h|^2\,dv_{g_0}: \ h\in C^\infty(\mathbb S^2, M)
\ \mbox{is a nontrivial harmonic map}\big\}.
$$
\item [iii)] for each $1\le i\le L$, there exist finitely many nontrivial harmonic maps $\omega_{ij}:\mathbb S^2\to M$,
$1\le j\le L_i$ with $L_i\le \frac{\E(0)}{\alpha(M)}$, such that
\begin{equation}\label{EIfinitetime}
\lim_{t\uparrow T_i}\E(A(t), \phi(t))=\E(A(T_i), \phi_i(T_i))+\sum_{j=1}^{L_i}\int_{\mathbb S^2}|\nabla\omega_{ij}|^2\,dv_{g_0}.
\end{equation}
\item[iv)] a) there exist $t_i\uparrow\infty$, a sequence of gauge transformations $\{s_i\}\subset\G_{2,2}$,
a set of finitely many points $S:= \{x_1, \cdots, x_{k_0}\}\subset \Sigma$, with $k_0\le \frac{\E(0)}{\alpha(M)}$,
 and  a Yang-Mills-Higgs field $(A_\infty,\phi_\infty)\in\A\times\S$ such that
 $s_i^*A(t_i)\rightarrow A_\infty$ in $H^1(\Sigma)$, $s_i^*\phi(t_i)\rightharpoonup \phi_\infty$ in $H^1(\Sigma)$,
 and $s_i^*\phi(t_i)\rightarrow \phi_\infty$
 in $H^1_{loc}(\Si\setminus S)$, as $i\rightarrow \infty$; b) there exist finitely many nontrivial harmonic maps $\omega_{p}: \mathbb S^2 \to M, 1\le p\le p_0\le \frac{\E(0)}{\alpha(M)}$ such that
\begin{equation}\label{EIinfinitytime}
  \lim_{i\to \infty} \E(A(t_i), \phi(t_i)) = \E(A_\infty, \phi_\infty) + \sum_{p=1}^{p_0}\E(\omega_{p});
\end{equation}
and  c) the images of $\{\omega_{p}\}_{p=1}^{p_0}$ and that of $\phi_\infty$ are connected.
\end{itemize}
\end{thm}

It is an interesting question whether the global weak solution given by Theorem \ref{main} is unique in a certain class
of weak solutions, which we plan to investigate in the future.

The paper is organized as follows. In section 2, we review some preliminary facts on connections and curvatures on
G-bundles. In section 3, we utilize the DeTurck's trick to obtain the local existence of unique smooth solutions to
the YMH  flow \eqref{e:heat}. In section 4, we establish a priori estimates of smooth solutions to \eqref{e:heat}
under the smallness condition. In section 5, through smooth approximations of $H^1$-initial data and applications of
the results from section 3 and section 4, we show the existence of local weak solutions to \eqref{e:heat} under $H^1$-initial data.
In section 6 and section 7, by extending the bubbling analysis for  YMH fields to approximate YMH fields with $L^2$-controlled tension fields, we obtain both the energy identity at finite singular time and asymptotic behavior at time infinity for the global
weak solutions to \eqref{e:heat}. In section 8, by utilizing Uhlenbeck's gauge fixing techniques, we are able to extend
the existence of local weak solutions to global weak solutions to \eqref{e:heat}.  The details of proof of Theorem \ref{main} are included in that of Theorem 5.1, Theorem 7.1, Theorem 7.2, and Theorem 8.3.

\section{Preliminaries}

First we recall some basic definitions of connection{\bf s} on bundle{\bf s}. Let $Ad P := P{\times_{Ad}}\ \g$ be the adjoint bundle of the principal $G$-bundle $P$. The space $\A$ of connections on $P$ is an affine space modelled on $\Om^1(Ad P)$,
here $\Om^k(Ad P)$ denotes the space of $Ad P$-valued $k$-forms for $k\ge 1$. Namely, if we fix a reference connection $D_{ref}$
on $P$, then
\[ \A = \Big\{D_A = D_{ref} + a\ |\ a\in \Om^1(Ad P)\Big\}. \]
Any connection $A\in \A$ gives rise to a covariant derivative $\nabla_A$ and an exterior differential operator $D_A$, which is the anti-symmetric part of $\nabla_A$, on the principal $G$-bundle $P$. The curvature of a connection
$A\in\A$ is defined by
\[ F_A := D_A\circ D_A \in \Om^2(Ad P). \]
The covariant derivative $\nabla_A$ can be extended to the associated bundle $\F$ as follows. Given a local trivialization of $\F$, let $\psi: \F|_U \to U\times M$ be the coordinate map. Then a section $\phi\in \S$ is locally equivalent to a map $u\in C^\infty(U, \F)$ and the connection $A$ can be written as
$\displaystyle D_A \big|_U = d + A$,
where $A := A_\al dx^\al$ is a $\g$-valued $1$-form. The covariant derivative $\nabla_A\phi$ is then defined by
$$\nabla_A\phi\big|_U := du + A\cdot u,$$
where the action of $A$ on the map $u$, $A\cdot u$, is induced by the symplectic action of $G$ on $M$.
More precisely, by the action of $G$ on $M$, every $\xi \in \g$ corresponds to a symplectic vector field $X_\xi$ by
\[ X_\xi(y) := \frac{d}{dt}\Big|_{t=0} exp(\xi t)\cdot y, ~\forall y \in M. \]
Then the action of the $\g$-valued $1$-form $A$ on $u$ is given by
\[ A\cdot u = A_\al\cdot u dx^\al = X_{A_\al}(u)dx^\al. \]
The operation of covariant derivative $\nabla_A$ on the (vertical) tangent bundle $T\F^v$ is slightly different. Denote the Levi-Civita connection induced by the metric $h$ on $M$ by $\nabla$, then we can define $\nabla_A:\Ga(T\F^v) \to \Ga(T\F^v\times T^*\Si)$ by
\[ \nabla_A V = \nabla V + A\cdot V:= \nabla V + \nabla_V X_{A_\al}\cdot dx^{\al}, \ V\in\Ga(T\F^v). \]
The covariant derivative $\nabla_A$ also extends to $T\F^v$-valued $p$-forms. Namely, we can define $\nabla_A: \Om^p(T\F^v) \to \Ga(T\F^v\otimes \Om^p(\Si)\otimes T^*\Si)$ by
\[ \nabla_A(\eta \otimes \om) = \nabla_A \eta \otimes \om + \eta \otimes \nabla \om,
\ \eta \otimes \om \in \Om^p(T\F^v), \]
where $\nabla$ is the Levi-Civita connection on $\Si$. The exterior derivative $D_A$ is defined
through the projection $\Om^p(T\F^v)\otimes T^*\Si \to \Om^p(T\F^v)$ in the standard way.

We will also need the following basic facts. For any connection $A$, we have the first Bianchi identity:
\begin{equation}\label{e:bianchi1}
  D_A F_A = 0,
\end{equation}\label{e:bianchi2}
and the second Bianchi identity:
\begin{equation}
  D_A^*D_A^* F_A = 0.
\end{equation}

There are two Laplace operators for the connection $A$ on fiber bundles. Namely, the Hodge Laplacian
$$\Delta_A = D_A^*D_A + D_AD_A^*,$$
and the rough Laplacian $\nabla_A^*\nabla_A$. The well-known Weitzenb\"ock formula describes the difference of these two Laplace operators on vector bundle valued forms. For example, the Weitzenb\"ock formula for $\Phi\in \Om^p(Ad P)$ is
\begin{equation}\label{e:weitzenbock}
  \nabla_A^*\nabla_A \Phi = \Delta_A \Phi + F_A\#\Phi + R_\Si\#\Phi,
\end{equation}
where $R_\Si$ is the Riemannian curvature of the base manifold $\Si$ and $\#$ denotes a multi-linear map with smooth coefficients.
Note that for a fiber bundle where the fiber $M$ is a Riemannian manifold with curvature tensor $R_M$, an extra term emerges in the Weitzenb\"ock formula. More precisely, for a section $\phi\in \S$, there is a pull-back bundle $\phi^*T\F$ on $\Si$. The curvature on $\phi^*T\F$ contains not only $F_A$ but also the pull-ball curvature $\phi^*R_M$. Therefore, for a section $\Psi \in \Om^p(\phi^*T\F)$, we have
\begin{equation}\label{e:weitzenbock1}
  \nabla_A^*\nabla_A \Psi = \Delta_A \Psi + F_A\#\Psi + R_\Si\#\Psi + R_M\#d\phi\#d\phi\#\Psi.
\end{equation}
This is the case when we apply this formula on $\Psi = D_A\phi$.

\medskip
\noindent\emph{Notations}: For simplicity, we will omit the subscription $A$ and simply use $D, F, \nabla, \Delta$ instead of $D_A, F_A, \nabla_A, \Delta_A$ if no confusions may occur.

\section{Local smooth solutions}

In this section, we will show both the existence and uniqueness of  local smooth solutions to the YMH  flow (\ref{e:heat}) for any smooth initial data $(A_0, \phi_0)$.
First, fix a smooth reference connection. For example, we may choose the initial connection $A_0$ as the reference connection. Then any connection $D$ corresponds to a 1-form $a\in \Om^1(Ad P)$ by
$$D = A_0 + a. $$
The curvature of $A$ is
\[ F_A = D_A\circ D_A = F_{A_0} + D_0a + a\wedge a = F_{A_0} + D_A a - a\wedge a, \]
since $D_A a = D_{A_0}a + [a, a]$. Then the equation (\ref{e:heat}) can be written as
\begin{equation}\label{e:heat01}
\left\{
\begin{aligned}
\frac{\partial a}{\partial t} &= -D_A^*D_Aa - D_A^*F_{A_0} + D_A^*(a\wedge a) - \phi^*D_A\phi, \\
\frac{\partial\phi}{\partial t} &= -D_A^*D_A\phi - (\mu(\phi) - c)\cdot \nabla\mu(\phi).
\end{aligned}
\right.
\end{equation}
It is well-known that that the first equation of $A$ in the system (\ref{e:heat01}) is degenerate,
since the Yang-Mills functional is invariant under gauge transformations.
To overcome this difficulty, we adapt DeTurck's trick and consider a
gauged system equivalent to \eqref{e:heat01} that is parabolic.
Our main result in this section is the following.

\begin{thm}\label{t:smooth}
  For any smooth initial data $(A_0, \phi_0)\in \A\times \S$, there exist a $T>0$ and a unique smooth solution
  $(A,\phi)$ to the Yang-Mills-Higgs  flow equation~(\ref{e:heat}) in $\Sigma\times [0, T)$,
  with $\displaystyle (A,\phi)\big|_{t=0}=(A_0,\phi_0)$.
\end{thm}
\begin{proof}
We consider the following perturbed system for $\overline {A}$ and $\bar{\phi}$:
\begin{equation}\label{e:heat-DeTurck}
\left\{
\begin{aligned}
\frac{\partial\bar{a}}{\partial t} &= -\overline{D}^*\overline{F} - \bar{\phi}^*\overline{D}\bar{\phi} -\overline{D}\overline{D}^*\bar{a}, \\
\frac{\partial\bar{\phi}}{\partial t} &= -\overline{\nabla}^*\overline{D}\bar{\phi} - (\mu(\bar{\phi}) - c)\cdot \nabla\mu(\bar{\phi}) + (\overline{D}^*\bar{a})\bar{\phi},
\end{aligned}
\right.
\end{equation}
under the initial condition $(\bar{a}(0), \bar{\phi}(0)) = (0, \phi_0)$, where
$$\overline{D} = D_{\overline{A}}, \ \bar{a} = \overline{D} - A_0,\ \overline{F} = F_{\overline{A}},
\ \overline{\nabla} = \nabla_{\overline{A}}.$$
If we denote the Hodge Laplacian on $\Om^p(Ad P)$ by $\overline{\De} := \overline{D}^*\overline{D} + \overline{D}\overline{D}^*$ and the Laplace-Beltrami operator on $\Ga(\F)$ by $\overline{\De}' :={\rm{tr}} \big(\overline{\nabla}^2\big)$, then the above system can be written as follows:
\begin{equation}\label{e:heat-DeTurck1}
\left\{
\begin{aligned}
\frac{\partial{\bar{a}}}{\partial t} + \overline{\De}\bar{a} &= -\overline{D}^*F_0 + \overline{D}^*(\bar{a}\wedge \bar{a}) - \bar{\phi}^*\overline{D}\bar{\phi}, \\
\frac{\partial{\bar{\phi}}}{\partial t} - \overline{\De}'\bar{\phi} &= - (\mu(\bar{\phi}) - c)\cdot \nabla\mu(\bar{\phi}) + (\overline{D}^*\bar{a})\bar{\phi},
\end{aligned}
\right.
\end{equation}
under the initial condition $(\bar{a}(0), \bar{\phi}(0)) =(0, \phi_0)$.
Note that the Hodge Laplacian $\overline\De$ and the Laplace-Beltrami operator $\overline\De'$ differs by a sign when applied to $\Om^0(\F)$. Then one can verify that the system (\ref{e:heat-DeTurck1}) is a quasilinear parabolic system. The standard parabolic theory implies that
for any smooth initial data $(A_0, \phi_0)$, there exists a unique smooth solution $(\bar{a}, \bar{\phi})\in C^\infty([0,T), \A\times \S)$ to (\ref{e:heat-DeTurck1})
for some $0<T=T(A_0,\phi_0)\le +\infty$ (see for example~\cite{FH} for a proof).

Next we choose a family of gauge transformations $\{S(t)\}_{0\le t<T}$,  which satisfies the following ordinary differential equation:
\begin{equation}\label{e:gauge-DeTurck}
\left\{
\begin{aligned}
\dt{S} &= -S\circ \overline{D}^*\bar{a}, \\
S(0) &= {\rm {id}}.
\end{aligned}
\right.
\end{equation}
We claim that the pair $(A(t), \phi(t)) = S(t)^*(\overline{A}(t), \bar{\phi}(t))$, $0\le t<T$, is a solution to the YMH
 flow equation
(\ref{e:heat}). Indeed, since
\[ D = S\circ \overline{D} \circ S^{-1}, ~ \dt{S^{-1}} = -S^{-1}\circ\dt{S}\circ S^{-1},\]
we have, by (\ref{e:gauge-DeTurck}),
\begin{equation*}
    \begin{split}
       \frac{\partial{a}}{\partial t} &= \dt{S}\circ \overline{D} \circ S^{-1} + S\circ \frac{\partial{\overline{D}}}{\partial t} \circ S^{-1}
        + S\circ \overline{D} \circ \dt{S^{-1}} \\
         &= S\circ (-\overline{D}^*\bar{a}\circ \overline{D} -\overline{D}^*\overline{F} - \bar{\phi}^*\overline{D}\bar{\phi}
         - \overline{D}\overline{D}^*\bar{a} + \overline{D}\circ \overline{D}^*\bar{a}) \circ S^{-1} \\
         &= S\circ (-\overline{D}^*\overline{F} - \bar{\phi}^*\overline{D}\bar{\phi}) \circ S^{-1} \\
         &= -D_A^*F_A - \phi^* D_A\phi,
     \end{split}
\end{equation*}
where we have used the identity:
\[ \overline{D}\circ \overline{D}^*\bar{a} = \overline{D}\overline{D}^*\bar{a} + \overline{D}^*\bar{a}\circ \overline{D}. \]
Also, since $\phi = S\circ\bar{\phi}$, we have
\begin{equation*}
    \begin{split}
       \frac{\partial{\phi}}{\partial t} &= \dt{S}\circ\bar{\phi} + S\circ \frac{\partial{\bar{\phi}}}{\partial t} \\
         &= -S\circ (\overline{D}^*\bar{a})\circ \bar{\phi} + S\circ\big(-\overline{\nabla}^*\overline{D}\bar{\phi} - (\mu(\bar{\phi}) - c)\cdot \nabla\mu(\bar{\phi})
         + (\overline{D}^*\bar{a})\bar{\phi}\big)  \\
         &= S\circ \big(-\overline{\nabla}^*\overline{D}\bar{\phi} - (\mu(\bar{\phi}) - c)\cdot \nabla\mu(\bar{\phi})\big) \\
         &= -\nabla_A^*D_A\phi - (\mu(\phi) - c)\cdot \nabla\mu(\phi).
     \end{split}
\end{equation*}
Moreover, the initial condition is preserved under the gauge transformation. Therefore we obtain a local smooth solution
$(A,\phi)\in C^\infty([0,T), \A\times \S)$ to the original YMH  flow equation (\ref{e:heat}) by solving the gauged parabolic system (\ref{e:heat-DeTurck1})
and the gauge transformations (\ref{e:gauge-DeTurck}).

The uniqueness part can be shown from the observation that we can reverse the above process and obtain a local smooth solution to (\ref{e:heat-DeTurck}) from a local smooth solution to the YMH flow equation (\ref{e:heat}) through the gauge transformations.  In fact, suppose $(A, \phi)\in
C^\infty([0, T), \A\times \S)$ is a solution to the YMH  flow equation (\ref{e:heat}), then we first solve the following equation
 to get a family of gauge transformations:
\begin{equation}\label{e:gauge-DeTurck1}
\left\{
\begin{aligned}
\dt{S} &= -D^*D S - D^*\circ Sa, \\
S(0) &= {\rm{id}}.
\end{aligned}
\right.
\end{equation}
Note that this is a parabolic system of $S$. Indeed, write $D = D_0 + a$, we have
\begin{equation*}
  \begin{aligned}
    D^*D S &= D_0^*D_0S + D_0^*(aS) + a^*(D_0S) + a^*aS\\
               &= D_0^*D_0S + (D_0^*a)S + 2a^*(D_0S) + a^*aS.
  \end{aligned}
\end{equation*}
Since $S$ can be viewed as a 0-form, we have $D_0^*S = 0$ and hence $\Delta_0 S = D_0^*D_0S$.
Thus the equation (\ref{e:gauge-DeTurck1}) is a parabolic system and has a unique local smooth
solution $S\in C^\infty([0,T))$.

A straightforward calculation shows that the pair $(\overline{A}(t), \bar{\phi}(t)) = (S^{-1})^*(A(t), \phi(t))$,
$0\le t<T$, is a solution to the system (\ref{e:heat-DeTurck}). In fact, the gauge transformation $S$ we get from (\ref{e:gauge-DeTurck1}) is exactly
 the same gauge transformation we obtain from (\ref{e:gauge-DeTurck}). To see this, note that since $\overline{D}^* = (S^{-1})^*(D^*) = S^{-1}\circ D^*\circ S$ and
$$\bar{a} = \overline{D} - A_0 = S^{-1}\circ D \circ S - D + a, $$
there holds
\[ S\circ \overline{D}^*\bar{a} = D^*\circ S\bar{a} = D^*\circ(D\circ S - S\circ D + Sa) = D^*D S + D^*\circ Sa. \]
Thus the equation (\ref{e:gauge-DeTurck}) and (\ref{e:gauge-DeTurck1}) are identical.

Now we can complete the proof of uniqueness. Suppose that we have two local smooth solutions $(A_1, \phi_1), (A_2, \phi_2)\in C^\infty([0, T),\A\times\S)$
of the YMH flow equation (\ref{e:heat}), then one solves the equation (\ref{e:gauge-DeTurck1}) to get two smooth gauge transformations $S_1(t), S_2(t)$,
$0\le t<T$,  and hence two local smooth solutions $(\bar{a}_1, \bar{\phi}_1), (\bar{a}_2, \bar{\phi}_2)$ to the system (\ref{e:heat-DeTurck}) with same initial value. By the uniqueness theorem of (\ref{e:heat-DeTurck}), we have $(\bar{a}_1, \bar{\phi}_1) = (\bar{a}_2, \bar{\phi}_2)$, $0\le t<T$. On the other hand,
since the parabolic system (\ref{e:gauge-DeTurck1}) is equivalent to the equation (\ref{e:gauge-DeTurck}), it follows from
the uniqueness of ordinary differential equations that $S_1 = S_2$. This in turn implies that $(A_1, \phi_1)=(A_2, \phi_2)$ for
$0\le t<T$. This completes the proof.
\end{proof}

\section{A priori estimates}

In this section, we will derive some a priori estimates of the local smooth solution obtained in the last section.

For $0<T\le +\infty$,
suppose $(A, \phi) \in C^\infty([0, T), \A)\times C^\infty([0, T), \S)$ is a smooth solution to the YMH  flow
equation (\ref{e:heat}), with the initial value $(A_0, \phi_0) \in \A\times \S$.
We first derive the evolution equation for the curvature $F_A$ of $A$ under the YMH flow
equation (\ref{e:heat}). By the definition of $F$ and a direct calculation, we have
\[  \frac{\partial{F_A}}{\partial t} = \frac{\partial}{\partial t}{}(D_A\circ D_A)
= \frac{\partial D_A}{\partial t}\circ D_A + D_A\circ \frac{\partial{D_A}}{\partial t}
= D_A(\frac{\partial{A}}{\partial t}). \]
Hence, under the YMH  flow equation (\ref{e:heat}), $F_A$ satisfies:
\begin{equation}\label{e:curvature}
\frac{\partial{F_A}}{\partial t} = - D_A D_A^* F_A - D_A(\phi^* D_A \phi).
\end{equation}
Note that it is not hard to see
\[ D_A(\phi^* D_A \phi) = D_A\phi\#D_A\phi + \phi^* F_A\phi. \]
The equation (\ref{e:curvature}) is a parabolic equation,  since by the first Bianchi identity (2.1), we have
\[ \Delta F_A = D_A D_A^* F_A. \]
Applying the Weitzenb\"ock formula (\ref{e:weitzenbock}), we can rewrite (\ref{e:curvature}) as
\begin{equation}\label{e:curvature1}
  \frac{\partial{F_A}}{\partial t} = -\nabla_A^*\nabla_A F_A - R_\Si\#F_A - F_A\#F_A - D_A\phi\#D_A\phi - \phi^*F_A\phi.
\end{equation}

Now we set  the energy density of $(A(t),\phi(t))$ by
\[ e(A(t), \phi(t)) := |F_A|^2+|D_A\phi(t)|^2 + |\mu(\phi(t))-c|^2, \]
the (total) Yang-Mills-Higgs energy of $(A(t),\phi(t))$ by
$$\E(t) := \E(A(t), \phi(t))=\int_\Sigma e(A(t), \phi(t))\,dv_g,$$
and the local Yang-Mills-Higgs energy of $(A(t),\phi(t))$ by
$$\E(t, B_R(x)) := \int_{B_R(x)}e(A(t),\phi(t))\,dv_g,$$
where $B_R(x):=\big\{y\in\Sigma: \ d(y,x)<R\big\}\subset\Sigma$ denotes the geodesic ball with center $x$ and  radius $R$,
and $d(\cdot, \cdot)$ denotes the distance function on ($\Sigma, g$).

Since the YMH  flow equation~(\ref{e:heat}) is the negative gradient flow of the YMH functional,
we have the following global
energy inequality.
\begin{lem}\label{l:energy-inequ}
{\rm{(energy inequality)}} For $0\le t<T$, it holds
\begin{equation}\label{e:energy-inequ}
  \E(t) + 2\int_0^t\int_\Si\Big(\abs{\frac{\partial{A}}{\partial t}}^2 + \abs{\frac{\partial{\phi}}{\partial t}}^2\Big)\, dv_gdt
  \le \E(0).
\end{equation}
\end{lem}
\begin{proof}
From the equation (\ref{e:heat}), we have
\begin{equation}\nn
\begin{aligned}
    \dt{\E} &= 2\int_\Sigma\Big( \< D_A^*F_A + \phi^* D_A\phi, \frac{\partial{A}}{\partial t}\> + \<\nabla_A^*\nabla_A\phi
    + (\mu(\phi) - c) \cdot\nabla\mu(\phi), \frac{\partial{\phi}}{\partial t}\>\Big) \,dv_g \\
    &= - 2\int_\Sigma\Big(\abs{\frac{\partial{A}}{\partial t}}^2 + \abs{\frac{\partial{\phi}}{\partial t}}^2\Big) \,dv_g.
\end{aligned}
\end{equation}
Integrating this inequality on $[0, t]$ yields (\ref{e:energy-inequ}).
\end{proof}

We also have the following local energy inequality.
\begin{lem}\label{l:energy-inequ-local}
{\rm{(local energy inequality)}} There exist $R_0>0$ and $C_0>0$ depending only $\Sigma$ such that for $0<t<T$ and $x\in\Sigma$,
it holds
\begin{equation}\label{e:energy-inequ-local}
\E(t, B_R(x)) \le \E(0, B_{2R}(x)) + \frac{Ct}{R^{2}}\E(0).
\end{equation}
\end{lem}
\begin{proof} Let $\eta \in C^\infty_0(B_{2R})$ such that $0\le \eta\le 1$, $\eta = 1$ on $B_R$,  and $|d\eta| \le \frac{C}{R}$. Then we
can calculate
\begin{equation*}
\begin{aligned}
  \frac12\dt{} \int_\Sigma e(A, \phi)\eta^2 \,dv_g
  =& \int_\Sigma \eta^2\< D_A\phi, \frac{\partial A}{\partial t}\phi \> \,dv_g
  + \int_\Sigma \eta^2\< D_A\phi, D_A\frac{\partial\phi}{\partial t} \> \,dv_g \\
  &+ \int_\Sigma\eta^2\< F_A, D_A\frac{\partial A}{\partial t}\> \,dv_g
  + \int_\Sigma\eta^2(\mu(\phi) - c)\cdot\nabla\mu(\phi)\cdot
  \frac{\partial\phi}{\partial t} \,dv_g.
\end{aligned}
\end{equation*}
Rearranging terms in the right side and integrating by parts, we have
\begin{equation*}
\begin{aligned}
  \frac12\dt{} \int_\Sigma e(A, \phi)\eta^2 \,dv_g
  =& \int_\Sigma\eta^2\< D_A^*D_A\phi + (\mu(\phi)-c)\cdot\nabla\mu(\phi), \frac{\partial\phi}{\partial t} \> \,dv_g\\
  &+ \int_\Sigma \eta^2\< D_A^*F_A + \phi^* D_A\phi, \frac{\partial A}{\partial t} \> \,dv_g \\
  &+ \int_\Sigma d(\eta^2)\< D_A\phi, \frac{\partial\phi}{\partial t} \> \,dv_g
  + \int_\Sigma d(\eta^2)\< F_A, \frac{\partial A}{\partial t}\> \,dv_g.
\end{aligned}
\end{equation*}
From the equation (\ref{e:heat}) and Young's inequality, we obtain
\begin{equation*}
\begin{aligned}
  \frac12\dt{} \int_\Sigma e(A, \phi)\eta^2 \,dv_g
  \le& - \frac12\int_\Sigma\eta^2\big(|\frac{\partial\phi}{\partial t}|^2 + |\frac{\partial A}{\partial t}|^2\big)  \,dv_g
  + \frac{C}{R^{2}}\int_\Sigma \big(|D_A\phi|^2 + |F_A|^2\big) \,dv_g.
\end{aligned}
\end{equation*}
Integrating on $[0,t]$, we arrive at
\begin{equation*}
\begin{aligned}
 & \E(t, B_{R}(x)) + \int_0^t\int_{B_R(x)}\Big(|\frac{\partial\phi}{\partial t}|^2 + |\frac{\partial A}{\partial t}|^2\Big)  \,dv_g \\
 & \le \E(0, B_{2R}(x))
  + \frac{C}{R^2}\int_0^t\int_{\Sigma}(|D_A\phi|^2 + |F_A|^2) \,dv_g dt.
  \end{aligned}
\end{equation*}
This and the energy inequality (\ref{e:energy-inequ}) now yield the local energy inequality (\ref{e:energy-inequ-local}).
\end{proof}

Next, we will derive a Bochner type formula for smooth solutions to the YMH flow equation (\ref{e:heat}).
In order to do this, we first set for any $\lambda>0$,
\[ \widehat{e}_\lambda(A, \phi) := \sqrt{\lambda+|F_A|^2}+|D_A\phi|^2 . \]
Recall that for any smooth section $\Phi$ of $\A\times\S$, there holds
\begin{equation}\label{e1}
  \frac12 \Delta_g|\Phi|^2 = -\< \nabla_A^*\nabla_A \Phi, \Phi\> + |\nabla_A\Phi|^2.
\end{equation}

\begin{lem}\label{l:bochner}
{\rm{(Bochner formula)}} Let $(A,\phi)\in C^\infty([0, T), \A\times \S)$ be a smooth solution of the YMH
flow equation (\ref{e:heat}).
There exists a constant $C>0$ depending only on $\Sigma,\  \A,$ and $\S$ such that
\begin{equation}\label{bochner1}
 \big(\frac{\partial}{\partial t} - \Delta_g\big)\widehat{e}_1(A, \phi)
 \le C\big(1 + |F_A| + |D_A\phi|^2\big)\widehat{e}_1(A, \phi)
\end{equation}
holds on $\Sigma\times (0, T)$.
\end{lem}
\begin{proof} First, applying the identity (\ref{e1}), we have
\begin{equation*}
  \frac12 \Delta_g|F_A|^2 = - \< \nabla_A^*\nabla_A F_A, F_A\> + |\nabla_A F_A|^2.
\end{equation*}
Second, applying the equation (\ref{e:curvature1}), we have
\begin{equation*}
\begin{aligned}
  \frac12 \frac{\partial}{\partial t} |F_A|^2
  &= -\< \nabla^*\nabla F_A + F_A\# F_A + R_\Si\# F_A, F_A\> - F_A\#D_A\phi\#D_A\phi - \< F_A\phi, F_A\phi\> \\
  &\le -\< \nabla^*\nabla F_A, F_A\> +C(1+|F_A| + |R_\Si|)|F_A|^2 + C|F_A||D_A\phi|^2.
\end{aligned}
\end{equation*}
Combining these two inequalities yields
\begin{equation}\label{e3}
  \frac12\big(\frac{\partial}{\partial t} - \Delta_g\big)|F_A|^2 + |\nabla_A F_A|^2 \le C(1+ |F_A|)|F_A|^2 + C|F_A||D_A\phi|^2.
\end{equation}
Direct calculations imply that
\begin{eqnarray}\label{bochner-lhs}
 \frac12\big(\frac{\partial}{\partial t} - \Delta_g\big)|F_A|^2
 &=& \frac12\big(\frac{\partial}{\partial t} - \Delta_g\big)\big(\sqrt{1+|F_A|^2}\big)^2\nonumber\\
 &=&\sqrt{1+|F_A|^2}\Big(\frac{\partial}{\partial t} - \Delta_g\Big)\sqrt{1+|F_A|^2}
 -\Big|\nabla\sqrt{1+|F_A|^2}\Big|^2.
 \end{eqnarray}
Recall that by Kato's inequality, it holds that
$$\big|\nabla |F_A|\big|\le \big|\nabla_A F_A\big|,$$
and hence we have that
\begin{eqnarray}\label{kato1}
\big|\nabla_A F_A\big|^2-\Big|\nabla\sqrt{1+|F_A|^2}\Big|^2
&\ge&\big|\nabla_AF_A\big|^2-\frac{|F_A|^2\big|\nabla|F_A|\big|^2}{1+|F_A|^2}\nonumber\\
&\ge& \big|\nabla_AF_A\big|^2-\big|\nabla|F_A|\big|^2\ge 0.
\end{eqnarray}
It is clear that by substituting (\ref{bochner-lhs}) into (\ref{e3}) and applying (\ref{kato1}), we obtain
\begin{equation}\label{e4}
  \big(\frac{\partial}{\partial t} - \Delta_g\big)\sqrt{1+|F_A|^2} \le C(1+ |F_A|)|F_A| + C|D_A\phi|^2.
\end{equation}
Next we will make a similar calculation on $|D_A\phi|^2$. By (\ref{e1}), we have
\begin{equation*}
  \frac12 \Delta_g|D_A\phi|^2 = - \< \nabla_A^*\nabla_A D_A\phi, D_A\phi\> + |\nabla_A D_A\phi|^2.
\end{equation*}
Applying the YMH  flow equation (\ref{e:heat}), we have
\begin{equation*}
  \begin{aligned}
    &\frac12 \frac{\partial}{\partial t} |D_A\phi|^2
    = \big\< D_A(\frac{\partial\phi}{\partial t}), D_A\phi\big\> + \big\< (\frac{\partial}{\partial t}{D_A})\phi, D_A\phi\big\>\\
    &= -\big\< D_A\big(D_A^*D_A\phi + \nabla h(\phi)\big), D_A\phi\big\>
         - \big\< \big(D_A^*F_A+\phi^*D_A\phi\big)\phi, D_A\phi \big\>\\
    &= -\big\< D_AD_A^*D_A\phi, D_A\phi\big\> - \nabla^2 h(\phi)\big(D_A\phi, D_A\phi\big)
         - \big\<(D_A^*F_A)\phi, D_A\phi\big\> - \big|\phi^*D_A\phi\big|^2.
  \end{aligned}
\end{equation*}
Here $h(\phi) := \frac12 \abs{\mu(\phi) - c}^2$. Applying the Weiztenb\"ock formula (\ref{e:weitzenbock1}), we have
\begin{equation*}
  \begin{aligned}
    &D_AD_A^*D_A\phi = \Delta_A D_A\phi - D_A^*D_AD_A\phi\\
    &= \nabla_A^*\nabla_A D_A\phi + F_A\#D_A\phi + R_\Si\# D_A\phi + R_M\#D_A\phi\#D_A\phi\#D_A\phi - D_A^*(F_A\phi)\\
    &= \nabla_A^*\nabla_A D_A\phi + F_A\#D_A\phi + R_\Si\# D_A\phi + R_M\#D_A\phi\#D_A\phi\#D_A\phi\\
    &\quad - (D_A^*F_A)\phi-F_A\#D_A\phi.
  \end{aligned}
\end{equation*}
Combining these two equations together, it is not hard to see that
\begin{equation}\label{e5}
\big(\frac{\partial}{\partial t} - \Delta_g\big) |D_A\phi|^2 \le -|\nabla_A D_A\phi|^2 + C(1+ |F_A|+|D_A\phi|^2)|D_A\phi|^2.
\end{equation}
Putting (\ref{e4}) together with (\ref{e5}) yields Bochner's formula (\ref{bochner1}). The proof is complete.
\end{proof}

Now we will prove an $\ep$-gradient estimate for smooth solutions $(A,\phi)$ of
the YMH  flow equation (\ref{e:heat}).
For any point $z :=(x,t)\in \Si\times(0, T)$, denote the parabolic ball, with center $z$ and radius $R>0$,  by
\[ \mathbb P_R(z) := \Big\{(y, s)\in \Sigma\times \Real \  \big |\ y\in B_R(x), \ t-R^2<s<t\Big\}.  \]

\begin{lem}\label{l:ep-reg}
{\rm{($\ep$-regularity)}}
There exist positive constants $\ep_0, R_0, C_0$ depending only on $\Sigma,\ \A$, and $\S$ such that for any
$z_0=(x_0, t_0)\in\Sigma\times (0, T)$ and $0<R<\min\big\{R_0, \sqrt{t_0}\big\}$, if
\begin{equation}\label{e:ep-condition}
  \sup_{t_0-R^2<t<t_0}\int_{B_{R}(x_0)}\big(|F_A| + |D_A\phi|^2\big)\,dv_g \le \ep_0,
\end{equation}
then for any $0<r<R$, it holds that
\begin{equation}\label{ep-reg}
 \sup_{z\in \mathbb P_r(z_0)} \big(|F_A| + |D_A\phi|^2\big)(z)\le \frac{C_0}{(R-r)^2}.
 \end{equation}
\end{lem}
\begin{proof} Let $i_0=i_0(\Sigma)>0$  denote the injectivity radius of $\Sigma$ and $0<R_0 < \min\{i_0, \sqrt{\epsilon_0}, 1\}$. Note that
$\F\big|_{B_{R_0}(x_0)}$ is  a trivial fiber bundle. For $0<R\le R_0$, define $v:[0,R]\to\mathbb R_+$ by
\[ v(r) := (R-r)^2\sup_{z\in \mathbb P_r(z_0)}\big(\sqrt{1+|F_A|^2}+|D_A\phi|^2\big)(z). \]
Assume that $v$ attains its maximum over $[0,R]$ at $r_0\in (0, R)$.  Let $z_1 := (x_1, t_1) \in \mathbb P_{r_0}(z_0)$ be
such that
\[ e_0 := \big(\sqrt{1+|F_A|^2}+|D_A \phi|^2\big)(z_1)
= \sup_{z\in \mathbb P_{r_0}(z_0)} \big(\sqrt{1+|F_A|^2}+|D_A\phi|^2\big)(z). \]
It is clear that for any $0\le r\le R$,
\begin{equation}\label{e6}
  v(r) \le v(r_0) = (R-r_0)^2\sup_{z\in \mathbb P_{r_0}(z_0)}\big(\sqrt{1+|F_A|^2}+|D_A \phi|^2\big)(z) = (R-r_0)^2 {e}_0.
\end{equation}
Observe that if $(R-r_0)^2 e_0\le 4$, then we would have
$$v(r)\le v(r_0)=(R-r_0)^2 e_0\le 4,\ \forall\ 0\le r\le R, $$
which implies (\ref{ep-reg}). Hence we may assume
\begin{equation}\label{radius_ass}
(R-r_0)^2 e_0>4.
\end{equation}
For $\lambda>0$, define the parabolic dilation $P_\lambda:\mathbb R^2\times \mathbb R\to\mathbb R^2\times\mathbb R$ by
$$P_\lambda(x,t)=(\lambda x, \lambda^2 t).$$
Denote $\rho_0 = {e}_0^{-\frac12}$ and define the rescaled pair $(\phi_{\rho_0}, A_{\rho_0})$ and the metric $g_{\rho_0}$ by
$$\phi_{\rho_0}(z) = \phi\big(z_1+ P_{\rho_0}(z)\big), \  A_{\rho_0}(z) = \rho_0A\big(z_1+ P_{\rho_0} (z)\big),
\ g_{\rho_0}(x)=g(x_1+\rho_0x),
$$
for $z\in \displaystyle \mathbb P_{\frac{r_0}{\rho_0}}\Big(P_{\rho_0^{-1}}(z_0-z_1)\Big)$
and $x\in \displaystyle B_{\frac{r_0}{\rho_0}}\Big(\frac{x_0-x_1}{\rho_0}\Big)$.
It is easy to check that
\[ \big|F_{A_{\rho_0}}\big|^2(z) = \rho_0^4\big|F_A\big|^2(P_{\rho_0}(z)), \  \big|D_{A_{\rho_0}}\phi_{\rho_0}\big|^2(z)
= \rho_0^2\big|D_A\phi\big|^2(P_{\rho_0}(z)). \]
Note that (\ref{radius_ass}) is equivalent to
\begin{equation}\label{e60}
\rho_0< \frac{R-r_0}2.
\end{equation}
Observe that since
\begin{eqnarray*} \widehat{e}_{\rho_0^4}(A_{\rho_0}, \phi_{\rho_0})(z)
&=& \Big(\sqrt{\rho_0^{4}+|F_{A_{\rho_0}}|^2}+\big|D_{A_{\rho_0}}\phi_{\rho_0}\big|^2\Big)(z)\\
&=&\rho_0^2\widehat{e}_1(\phi, A)(z_1+P_{\rho_0}(z)),
\end{eqnarray*}
for all $z\in \mathbb P_{\frac{r_0}{\rho_0}}\Big(P_{\rho_0^{-1}}(z_0-z_1)\Big)$,
it follows that
\[ \widehat{e}_{\rho_0}(A_{\rho_0}, \phi_{\rho_0})(0)
= \rho_0^2\widehat{e}_1(A, \phi)(z_1)=e_0^{-1}e_0= 1. \]
From (\ref{e6}), we have that $\mathbb P_{\rho_0}(z_1) \subset \mathbb P_{\frac{R+r_0}{2}}(z_0)$ and
hence by the definition of $v(r)$, it holds
\begin{equation}\label{grad_bound}
\begin{aligned}
  \sup_{z\in\mathbb P_1(0)}\widehat{e}_{\rho_0^4}(A_{\rho_0}, \phi_{\rho_0})(z)&
  = \rho_0^2\sup_{z\in \mathbb P_{\rho_0}(z_1)}\widehat{e}_1(A, \phi)(z) \\
  &\le e_0^{-1}\sup_{z\in\mathbb P_{\frac{R+r_0}{2}}(z_0)}
  \widehat{e}_1(A, \phi)(z) = e_0^{-1}\big(\frac{R-r_0}{2}\big)^{-2} v\big(\frac{R+r_0}{2}\big)\\
  &\le e_0^{-1}\cdot \big(\frac{R-r_0}{2}\big)^{-2}(R-r_0)^2 e_0 = 4.
\end{aligned}
\end{equation}
From Bochner's formula (\ref{bochner1}) for $(\phi, A)$ and straightforward calculations, we have
that
\begin{eqnarray} \big(\frac{\partial}{\partial t} - \Delta_{g_{\rho_0}}\big)
\widehat{e}_{\rho_0^4}\big(A_{\rho_0}, \phi_{\rho_0}\big)(z)
&=&\rho_0^4\Big[\big(\frac{\partial}{\partial t} - \Delta_{g}\big)
\widehat{e}_{1}\big(A, \phi\big)\Big](P_{\rho_0}(z))\nonumber\\
&\le& C\rho_0^4 \big(1+|F_A|+|D_A\phi|^2\big)\widehat{e}_1\big(A, \phi\big)(P_{\rho_0}(z))\nonumber\\
&=&C\big(\rho_0^2+|F_{A_{\rho_0}}|+|D_{A_{\rho_0}}\phi_{\rho_0}|^2\big)
\widehat{e}_{\rho_0^4}\big(A_{\rho_0}, \phi_{\rho_0}\big)(z)\nonumber\\
&\le& C\widehat{e}_{\rho_0^4}\big(A_{\rho_0}, \phi_{\rho_0}\big)(z),\label{subsolution}
\end{eqnarray}
where we have used (\ref{grad_bound}) and the fact $\rho_0\le 1$ in the last step.
Thus, by Moser's Harnack inequality for parabolic equations, we obtain
\begin{equation}
\begin{aligned}
  1=\widehat{e}_{\rho_0^4}\big(A_{\rho_0}, \phi_{\rho_0}\big)(0)
  &\le C \int_{\mathbb P_1(0)} \widehat{e}_{\rho_0^4}\big(A_{\rho_0}, \phi_{\rho_0}\big)= C\rho_0^{-2}\int_{\mathbb P_{\rho_0}(z_1)}\widehat{e}_1(A, \phi)\\
  &\le C \sup_{t_1 - \rho_0^2 < t <t_1} \int_{B_{\rho_0}(x_1)}\widehat{e}_1(A, \phi)\\
  &\le C\sup_{t_0-R^2<t<t_0}\int_{B_{R}(x_0)}(1+|F_A|+|D_A\phi|^2)\\
  & \le C(\ep_0+R_0^2)\le C\epsilon_0.
\end{aligned}
\end{equation}
This is clearly impossible if we choose a sufficiently small $\ep_0>0$. Thus (\ref{radius_ass}) doesn't hold, which implies
$$e_0^{-\frac12} \ge \frac{R-r_0}{2}$$
and hence $v(r) \le 4$ for all $0\le r\le R$. In particular, we have
\[ \sup_{z\in \mathbb P_r(z_0)}\widehat{e}_1(\phi, A)(z) = \frac{1}{(R-r)^2}v(r)\le \frac{4}{(R-r)^{2}}. \]
This completes the proof.
\end{proof}

Based on the $\epsilon$-regularity Lemma \ref{l:ep-reg},  we can derive point-wise estimates for higher order
derivatives of smooth solutions $(\phi, A)$ of the YMH  flow equation (\ref{e:heat}),
under the small energy condition (\ref{e:ep-condition}).

\begin{lem}\label{l:higher-reg} Under the same assumptions as in Lemma \ref{l:ep-reg},  there holds
\begin{equation}\label{e:higher-reg}
  \sup_{\mathbb P_{\frac{R}{2}}(z_0)}\big(R^{l+1}|\nabla^{l-1}_A F_A|+R^l|\nabla^l_A\phi|\big) <C\big(l,\epsilon_0\big)
\end{equation}
for any $l\in \Integer^+$.
\end{lem}
\begin{proof} Note that when $l=1$, (\ref{e:higher-reg}) follows directly from Lemma~\ref{l:ep-reg}.
By a simple scaling argument, we may assume $R=2$ and $z_0=(0,0)$.
In order to establish (\ref{e:higher-reg}) for all $l\ge 2$,  we need to show
\begin{equation}\label{e9}
\int_{\mathbb P_{1+2^{-l}}(0)} \big(|\nabla^{l-1}_AF_A|^2 + |\nabla_A^l\phi|^2\big) \le C(l,\epsilon_0)
\int_{\mathbb P_2(0)} e(\phi, A),
 \ \forall\ l\ge 2.
\end{equation}
From Lemma \ref{l:ep-reg}, we have the following estimate
\begin{equation}\label{grad_bound1}
\sup_{\mathbb P_{\frac32}(0)}\big(|F_A|+|D_A\phi|\big)\le C(\epsilon_0).
\end{equation}
We prove (\ref{e9}) by an induction on $l$.  Consider the case $l=2$,
we may assume for simplicity that $g|_{B_2(0)}$ is the euclidean metric.
Let $\eta\in C_0^\infty(B_2(0))$ such that $0\le \eta\le 1$, $\eta\equiv 1$ on $B_{\frac54}(0)$,
$\eta\equiv 0$ outside $B_{\frac32}(0)$, and $|\nabla\eta|\le 8$.
Multiplying (\ref{e:curvature1}) by $\eta^2F_A$, integrating the resulting equation over $B_2(0)$,
and applying integration by parts and H\"older's inequality, we obtain
\begin{equation}\label{l=2}
\begin{aligned}
  \frac12\dt{}\int_{B_2(0)}|F_A|^2\eta^2 &
  = -\int_{B_2(0)} \big(|\nabla_A F_A|^2\eta^2+\<\nabla_A F_A , F_A\nabla \eta^2\>\big)- \int_{B_2(0)}\eta^2\<R_\Si\#F_A, F_A\>\\
  &\quad - \int_{B_2(0)}\eta^2\<F_A\#F_A, F_A\>
  - \int_{B_2(0)} \eta^2\<D_A\phi\#D_A\phi+\phi^*F_A\phi, F_A\>\\
  &\le -\frac12\int_{B_2(0)}\eta^2|\nabla_A F_A|^2+
  C\int_{B_{\frac32}(0)}\big(|F_A|^2+|F_A|^3 + |D_A\phi|^4\big).
\end{aligned}
\end{equation}
By Fubini's theorem, we can find $t_*\in (-4, -(\frac54)^2)$ such that
\begin{equation}\label{fubini1}
\int_{B_2(0)} (|F_A|^2+|D_A\phi|^2)(t_*)\le C\int_{\mathbb P_2(0)}(|F_A|^2+|D_A\phi|^2).
\end{equation}
Integrating (\ref{l=2}) over $t\in [t_*, 0]$ and applying (\ref{fubini1}) and (\ref{grad_bound1}), we get
\begin{eqnarray}\label{l=2.1}
&&\int_{B_{\frac54}(0)} |F_A|^2(0)+\int_{t_*}^0\int_{B_{2}(0)}\eta^2|\nabla_A F_A|^2\nonumber\\
 &&\le\int_{B_{\frac54}(0)} |F_A|^2(t_*)+C\int_{\mathbb P_{\frac32}(0)}(|F_A|^2 +|F_A|^3 + |D_A\phi|^4)\nonumber\\
 &&\le C\int_{\mathbb P_2(0)}(|F_A|^2+|D_A\phi|^2)+C\int_{\mathbb P_{\frac32}(0)}(|F_A|^3 + |D_A\phi|^4)\nonumber\\
 &&\le C(\epsilon_0)\int_{\mathbb P_2(0)}\big(|F_A|^2+|D_A\phi|^2\big).
 \end{eqnarray}
To obtain a similar estimate for $\phi$,  we take $\nabla_A$ to both sides of the YMH 
flow equation (\ref{e:heat})$_2$ so that
it holds
\begin{equation}\label{heat2}
\begin{aligned}
  \frac{\partial}{\partial t}(\nabla_A\phi)
  &= -\nabla_A\nabla_A^*\nabla_A\phi - \nabla^2h(\phi)\cdot\nabla_A\phi + \frac{\partial A}{\partial t}\phi,
\end{aligned}
\end{equation}
where $h(\phi) = \frac12|\mu(\phi)-c|^2$.
Multiplying (\ref{heat2}) by $\eta^2\nabla_A\phi$, integrating the resulting equation on
$B_2(0)$, applying integration by parts and H\"older's inequality, and using the equation (\ref{e:heat})$_1$,
we obtain that, for any $0<\delta<1$,
\begin{equation*}
\begin{aligned}
\frac12\dt{}\int_{B_2(0)} |\nabla_A\phi|^2 \eta^2&\le -\int_{B_2(0)} \big(\eta^2|\nabla_A^*\nabla_A\phi|^2
+\<\nabla_A^*\nabla_A\phi, \nabla_A\phi \nabla\eta^2\>\big)\\
 &\quad+\int_{B_2(0)} \big|\nabla^2_\phi h(\phi)\big|\big|\nabla_A\phi\big|^2 \eta^2
 + \int_{B_2(0)}\big|\frac{\partial A}{\partial t}\big|\big|\nabla_A\phi\big|\eta^2\\
&\le -\frac12\int_{B_2(0)}\eta^2\big|\nabla_A^*\nabla_A\phi\big|^2
+ \delta\int_{B_2(0)}\eta^2\big|\frac{\partial A}{\partial t}\big|^2
+ C\delta^{-1}\int_{B_2(0)}\eta^2\big|\nabla_A\phi\big|^2\\
&\le -\frac12\int_{B_2(0)}\eta^2\big|\nabla_A^*\nabla_A\phi\big|^2
+ 2\delta\int_{B_2(0)}\eta^2\big|D_A^*F_A\big|^2
+ C\delta^{-1}\int_{B_2(0)}\eta^2\big|\nabla_A\phi\big|^2\\
&\le -\frac12\int_{B_2(0)}\eta^2\big|\nabla_A^*\nabla_A\phi\big|^2
+ C\delta\int_{B_2(0)}\eta^2\big|\nabla_AF_A\big|^2
+ C\delta^{-1}\int_{B_2(0)}\eta^2\big|\nabla_A\phi\big|^2.
\end{aligned}
\end{equation*}
Integrating this inequality over $[t_*,0]$ and applying (\ref{fubini1}),  we obtain
\begin{equation}\label{l=2.2}
\int_{B_2(0)} |\nabla_A\phi|^2(0)+ \int_{t_*}^0\int_{B_2(0)} \eta^2\big|\nabla_A^*\nabla_A\phi\big|^2
\le C\delta\int_{t_*}^0\int_{B_2(0)}\eta^2\big|\nabla_A F_A|^2 + C\delta^{-1}\int_{\mathbb P_2(0)}e(A, \phi).
\end{equation}
Adding (\ref{l=2.1}) and (\ref{l=2.2}) and choosing a sufficiently small $\delta>0$, we obtain
\begin{equation}\label{l=2.3}
\int_{B_2(0)} e(A, \phi)(0)+\int_{t_*}^0 \int_{B_2(0)} \eta^2(\big|\nabla_A^*\nabla_A\phi\big|^2+\big|\nabla_AF_A\big|^2\big)
\le C(\epsilon_0)\int_{\mathbb P_2(0)}e(A, \phi).
\end{equation}
Note that by integrating by parts, applying Ricci's identities for interchanging derivatives, and (\ref{grad_bound1}), we have
that
\begin{equation}\label{l=2.4}
\begin{aligned}
\int_{t_*}^0\int_{B_2(0)}\eta^2\big|\nabla^2_A\phi\big|^2 &\le 2\int_{t_*}^0\int_{B_2(0)}\eta^2\big|\nabla_A^*\nabla_A \phi\big|^2
+ C\int_{t_*}^0\int_{B_2(0)}\big(|\nabla\eta|^2|\nabla_A\phi|^2+\eta^2|\nabla_A\phi|^4\big)\\
&\le 2\int_{t_*}^0\int_{B_2(0)}\eta^2\big|\nabla_A^*\nabla_A \phi\big|^2
+ C(\epsilon_0)\int_{\mathbb P_2(0)}e(A, \phi),
\end{aligned}
\end{equation}
where the constant $C>0$ depends on the curvature $F_A$ and
the curvatures of $\Sigma$ and $M$.
Combining (\ref{l=2.3}) with (\ref{l=2.4}) yields (\ref{e9}) for $l=2$.

For any $k\ge 3$, we assume that (\ref{e9}) holds for all $l\le k$. We need to show that it also holds for $l = k+1$.
To do it, first apply the operator $\nabla_A^{k}$ to both sides of the YMH  flow equation (\ref{e:heat})$_2$ of $\phi$
the operator $\nabla_A^{k-1}$ to both sides of the YMH  flow equation (\ref{e:curvature1}). After interchanging the
order of derivatives, we obtain a system in the following form:
\begin{equation}\label{k-ymh}
\left\{
  \begin{aligned}
    \frac{\partial}{\partial t}\big(\nabla_A^{k}\phi\big) &= -\nabla_A^*\nabla_A^{k+1}\phi + \mathcal Q_1^k(\nabla_A\phi, F_A)
    +\nabla_A^{k}(\nabla h(\phi)), \\
    \frac{\partial}{\partial t}\big(\nabla_A^{k-1}F_A\big) &= -\nabla_A^*\nabla_A^{k} F_A + \mathcal Q_2^k(\nabla_A\phi, F_A)
    +\nabla_A^{k-1}(g(F_A, \phi)).
  \end{aligned}
  \right.
\end{equation}
where $\mathcal Q_1^k, \mathcal Q_2^k$ are lower order terms depending on derivatives of $\nabla_A\phi$ and $F_A$ up to order $k-1$, and
\[ g(F_A, \phi) = -R_\Si\#F_A -F_A\#F_A-D_A\phi\#D_A\phi-\phi^*F_A\phi. \]
More precisely, $\mathcal Q_1^k$ comes from changing $\nabla_A^k(\frac{\partial\phi}{\partial t})$  to $\frac{\partial}{\partial t}(\nabla_A^k \phi)$,
and $\nabla_A^k\nabla_A^*\nabla_A\phi$ to $\nabla_A^*\nabla_A\nabla_A^k\phi$.
If we denote the total curvature of the fiber bundle by $\widetilde{R}$ which involves $R_\Sigma, R_M$ and $F_A$, we have
\begin{equation}\label{e:Q1}
\mathcal Q_1^k(\nabla_A\phi, F_A) = \sum\nabla^{a}\widetilde{R}\big(\nabla_A^b \nabla_A\phi, \nabla_A^c \nabla_A\phi\big)\nabla_A^d
\nabla_A\phi,
\end{equation}
where the sum is taken over all indices $a, b, c, d\ge 0$ satisfying $a+b+c+d = k-1$. Then one can verify that
\begin{eqnarray}\label{e:Q11}
&&\Big|\sum\nabla_A^{a}\widetilde{R}\big(\nabla_A^b \nabla_A\phi, \nabla_A^c \nabla_A\phi\big)\nabla_A^d \nabla_A\phi\Big| \nonumber\\
&&\le C\sum\big|\nabla_A^{j_1} F_A\big|\cdots\big|\nabla^{j_r}_A F_A\big|\cdot\big|\nabla_A^{j_{r+1}}\nabla_A\phi\big|\cdots
\big|\nabla_A^{j_{r+s}}\nabla_A\phi\big|,
\end{eqnarray}
where $C>0$ is a constant depending on $R_\Sigma$ and $R_M$, and the indices satisfy
\begin{equation}\label{e:indices}
  0\le j_1, \cdots, j_{r+s}\le k-1; \quad j_1+\cdots+ j_{r+s} = k-1; \quad s\ge 3.
\end{equation}
Moreover, we have
\begin{equation}\label{e:h}
\big| \nabla_A^k(\nabla h(\phi))\big| \le C\sum |\nabla_A^{i_1}\phi|\cdots|\nabla_A^{i_p}\phi|
\end{equation}
where $C$ is a constant depending on $h$ and the indices $i_1,\cdots, i_p\ge 1$ satisfy $i_1+\cdots i_p = k$.

Thus multiplying (\ref{k-ymh})$_1$ by $\eta^2\nabla_A^k\phi$ and integrating on $B_2(0)$, we get
\begin{equation}\label{e:Q15}
\begin{aligned}
  &\frac{d}{dt}\int_{B_2(0)}\eta^2|\na^k\phi|^2 + \int_{B_2(0)}|\na^{k+1}\phi|^2\\
   &\le C\int_{B_2(0)}\eta^2|\na^{k}\phi|^2 + C\int_{B_2(0)}\eta^2|\mathcal Q_1^k||\na^k\phi| + C\int_{B_2(0)}\eta^2|\na^k(\nabla h(\phi))||\na^k\phi|\\
   &=: I + II + III.
\end{aligned}
\end{equation}
By (\ref{e:Q1}) and (\ref{e:Q11}), we have
\begin{equation}\label{e:Q12}
\begin{aligned}
  II \le &C\int_{B_2(0)}\eta^2|F_A|^r|\na\phi|^{s-1}|\na^k\phi|^2 + C\int_{B_2(0)}\eta^2|F_A|^{r-1}|\na\phi|^{s}|\na^{k-1}F_A||\na^k\phi|\\
  & +C\sum\int_{B_2(0)}\eta^2|\na^{j_1} F_A|\cdots|\na^{j_r} F_A|\cdot|\na^{j_{r+1}}\na\phi|\cdots|\na^{j_{r+s}}\na\phi||\na^k\phi|\\
  \le &C(\ep_0)\int_{B_2(0)}\eta^2(|\na^k\phi|^2+|\na^{k-1}F_A|^2)\\
  & +C\sum\int_{B_2(0)}\eta^2|\na^{j_1} F_A|^2\cdots|\na^{j_r} F_A|^2\cdot|\na^{j_{r+1}}\na\phi|^2\cdots|\na^{j_{r+s}}\na\phi|^2,
\end{aligned}
\end{equation}
where the sum in the last term is taken over all the indices satisfying (\ref{e:indices}) with $j_1,\cdots j_{r+s}\le k-2$.

Recall the Sobolev embedding for a section $\Psi$ in dimension 2
\begin{equation*}
\norm{\Psi}_{L^q}\le C\norm{\Psi}_{W^{1,2}} \le C\norm{\Psi}_{H^{1,2}}, ~\forall q>1
\end{equation*}
where  $\norm{\Psi}_{W^{1,2}}$ is the usual Sobolev norms of $|\Phi|$ and
\begin{equation*}
\norm{\Psi}_{H^{1,2}} := \norm{\Psi}_{L^2} + \norm{\na\Psi}_{L^2}
\end{equation*}
is the Sobolev norms with respect to $\na$. It follows that that for all $q>1, j\le k-2$, we have
\begin{equation*}
\norm{\na^{j}F_A}_{L^q}\le C\norm{\na^j F_A}_{H^{1,2}}\le C\sum_{i=1}^{k-1} \norm{\na^i F_A}_{L^2},
\end{equation*}
and
\begin{equation*}
\norm{\na^{j}\na\phi}_{L^q}\le C\norm{\na^{j+1}\phi}_{H^{1,2}} \le C\sum_{i=1}^{k-1} \norm{\na^{i+1}\phi}_{L^2}.
\end{equation*}
Applying this and the H\"older inequality in (\ref{e:Q12}), we get
\begin{equation}\label{e:Q13}
II\le C(\ep_0)\sum_{i=1}^{k-1}\left(\int_{B_2(0)}\eta^2(|\na^i F_A|^2 + |\na^{i+1}\phi|^2)\right)
\end{equation}
By (\ref{e:h}) and a similar argument, we have
\begin{equation}\label{e:Q14}
III\le C(\ep_0)\sum_{i=1}^{k-1}\left(\int_{B_2(0)}\eta^2|\na^{i+1}\phi|^2\right)
\end{equation}

Now inserting (\ref{e:Q13}) and (\ref{e:Q14}) back into (\ref{e:Q15}), we arrive at
\begin{equation*}
\frac{d}{dt}\int_{B_2(0)}\eta^2|\na^k\phi|^2 + \int_{B_2(0)}|\na^{k+1}\phi|^2 \le C(\ep_0)\sum_{i=1}^{k-1}\left(\int_{B_2(0)}\eta^2(|\na^i F_A|^2 + |\na^{i+1}\phi|^2)\right).
\end{equation*}
Consequently, integrating this inequality on $t$ as we did before and using the induction assumption, we obtain the desired bound
\begin{equation*}
\int_{t_*}^0\int_{B_2(0)}|\na^{k+1}\phi|^2 \le C(\ep_0)\sum_{i=1}^{k-1}\left(\int_{t_*}^0\int_{B_2(0)}\eta^2(|\na^i F_A|^2 + |\na^{i+1}\phi|^2)\right).
\end{equation*}

The estimate on $\na^{k}F_A$ can be achieved in the same way. Namely, $\mathcal Q_2^k$ emerges from interchanging the order of derivatives on $F_A$. Since we have
\begin{equation*}
\Big|\nabla_A^{k-1}\nabla_A^*\nabla_A F_A - \nabla_A^*\nabla_A^k F_A\Big|
\le C \sum_{i=0}^{k} \big|\nabla_A^i F_A\#\nabla_A^{k-i}F_A\big|,
\end{equation*}
it is easy to verify that the following estimate holds:
\begin{equation*}
\begin{aligned}
\Big|\mathcal Q_2^k(\nabla_A\phi, F_A)\Big|+\Big|\nabla^{k-1}_A(g(F_A,\phi))\Big|\le
& C \sum_{i=0}^{k} \left(|\nabla_A^i F_A||\nabla_A^{k-i}F_A| + |\nabla^{i+1}_A \phi||\nabla^{k-i}_A\phi|\right)\\
 &+ C\sum_{0\le j_1+j_2+j_3\le k} |\nabla_A^{j_1} F_A||\nabla_A^{j_2}\phi||\nabla_A^{j_3}\phi|.
\end{aligned}
\end{equation*}
Then multiplying the equation (\ref{k-ymh})$_2$ by $\nabla_A^{k-1}F_A$ and integrating by parts, we may obtain the desired bounds (\ref{e9}) for $l = k+1$.

Finally, (\ref{e9}) together with the parabolic Sobolev embedding theorems completes our proof.
\end{proof}

\begin{rem} {\rm  In dimension two, the assumption (\ref{e:ep-condition}) in Lemma \ref{l:ep-reg}
can be weaken to that there exists $R_0>0$ depending on $(\phi_0, A_0)$ and $\epsilon_0>0$ such that
\begin{equation}\label{e:ep-condition1}
\sup_{t_0-R_0^2\le t\le t_0}\int_{B_{R_0}(x_0)}|D_A\phi|^2\,dv_g\le \ep_0.
\end{equation}
In fact, it follows from H\"older's inequality and the energy inequality (\ref{e:energy-inequ}) that
$$\int_{B_{R_0}(x_0)}|F_A|\,dv_g\le {\rm{Vol}}(B_{R_0}(x_0))\Big(\int_{B_{R_0}(x_0)}|F_A|^2\,dv_g\Big)^\frac12
\le C \big(\mathcal E(t)\big)^\frac12R_0\le C\big(\mathcal E(0)\big)^\frac12 R_0\le \epsilon_0,$$
provided $\displaystyle R_0\le \frac{\epsilon_0}{C\sqrt{\mathcal E(0)}}$. }
\end{rem}

\section{Local weak solutions of the YMH  flow}

In this section, we will establish the local existence of
weak solutions to the initial value problem of the YMH flow equation~(\ref{e:heat}),
with any initial data $(A_0, \phi_0)\in \A_{1,2}\times \S_{1,2}$. For simplicity, we will denote, for $0<T\le +\infty$ and $1<p\le \infty, 1<q\le \infty$,
\begin{eqnarray*}
  L^p(L^q):=L^p([0,T), L^q(\Sigma)), \
  L^p(H^1):=L^p([0,T), H^1(\Sigma)), \
  H^1(L^q):=H^1([0,T), L^q(\Sigma)).
\end{eqnarray*}

Now we state the theorem on the local existence of weak solutions.
\begin{thm}\label{t:local}
Assume $(A_0, \phi_0)\in \A_{1,2}\times\S_{1,2}$,
 there exist $0<T_0\le +\infty$ and a weak solution $(A, \phi)$ to
  the YMH  flow equation~(\ref{e:heat}) and the initial condition (\ref{IVP}) on the interval $[0, T_0)$. Moreover, $(A,\phi)$ enjoys the properties that
  $F_A\in L^\infty(L^2)$, $\phi\in L^\infty(H^1)$, and there exists a gauge transformation $s\in \G_{1,2}$
  such that $(s^*A,s^*\phi)\in C^\infty\big((0,T_0), \A\times \S\big)$. Finally, if $0<T_0<+\infty$ is the maximal time interval for the weak solution,
  then
  \begin{equation}\label{local-ymh-concen}
  \limsup_{t\uparrow T_0} \max_{x\in\Sigma}\int_{B_R(x)} e(A(t), \phi(t))\,dv_g\ge \alpha(M), \ \forall\ R>0,
  \end{equation}
  where $\alpha(M)>0$ is a positive constant  given by
  \begin{equation}\label{energy-gap}
  \alpha(M):=\inf\Big\{\int_{\mathbb S^2} |\nabla h|^2\,dv_{g_0}: \ h\in C^\infty(\mathbb S^2, M)\ {\rm{is\ a\ nontrivial\ harmonic\ map}} \Big\}.
  \end{equation}
 \end{thm}
\begin{proof} We divide the proof into five steps:

\smallskip
\noindent{\it Step 1: Smooth approximation of initial data.}  Since $P\times_G M$ is a smooth manifold and ${\rm{dim}}(\Sigma)=2$, it follows from Schoen-Uhlenbeck's
density Lemma (cf. \cite{SU}), a local trivialization of the bundle $\F$, and the partition of unity that there exists a sequence of smooth sections
$\phi_0^n\in \S$ such that $\phi_0^n\rightarrow \phi$ in $H^1(\Sigma)$. It is standard that
there is a sequence of smooth sections $A_{0}^n\in \A$ such that $A_{0}^n\rightarrow A_0$ in $H^1(\Sigma)$.
Thus we may assume a uniform energy bound of $(A_0^n,\phi_0^n)$:
  \begin{equation}\label{e10}
    \E_n(0): = \E\big(A_{0}^n, \phi_{0}^n\big) \le C (=1+\E(0)).
  \end{equation}
Applying Theorem~\ref{t:smooth}, we conclude that there exist $T_n>0$ and a unique smooth solution $(A_n, \phi_n)\in C^\infty([0,T_n), \A\times \S)$
to the YMH heat flow equation~(\ref{e:heat}) and the initial condition (\ref{IVP}) with $(A_0,\phi_0)$ replaced by $(A_{0}^n, \phi_{0}^n)$.
We may assume that $T_n>0$ is the maximal time interval for $(A_n,\phi_n)$.

\smallskip
\noindent{\it Step 2: A uniform lower bound of $T_n$}. Let $\ep_0$ and $R_0$ be the constants given by Lemma~\ref{l:ep-reg}.
Let $\ep_0>0$ and $R_0>0$ be the constants given by lemma~\ref{l:ep-reg}. Since $(A_{0}^n, \phi_{0}^n)\rightarrow (A_0, \phi_0)$ in $H^1(\Sigma)$,
  there exists a uniform $R_1\in (0, R_0)$ independent of $n$ such that
  \begin{equation}\label{e11}
    \E_n(0, 2R_1): =\max_{x\in\Sigma} \int_{B_{2R_1}(x)}e(A_{0}^n, \phi_{0}^n)\,dv_g \le \frac{\ep_0}{2}.
  \end{equation}
Then by Lemma~\ref{l:energy-inequ-local}, there exists $\theta_0\in (0,1)$, depending only on $\Sigma$, $\E(0)$ and $R_1$, such that
for $T_n^1=\min\big\{T_n, \theta_0 R_1^2\big\}>0$ there holds
\begin{eqnarray}\label{e12}
\E_n(t, {R_1}):=\max_{x\in\Sigma} \int_{B_{R_1}(x)}e\big(A_n(t), \phi_n(t)\big)\,dv_g
&\le& \E_n(0, 2R_1) + CT_n^1R_1^{-2}\L_n(0)\nonumber\\
&\le& \frac{\epsilon_0}2+CT_n^1R_1^{-2}\big(1+\E(0)\big)\nonumber\\
&\le& \ep_0,
\end{eqnarray}
for all $0\le t\le T_n^1$.

Now we want to show $T_n\ge \theta_0 R_1^2$. Suppose, otherwise,  $T_n<\theta_0R_1^2$. Then we have
$T_n=T_n^1<\theta_0R_1^2$ so that (\ref{e12}) holds for $0\le t\le T_n$. Applying Lemma \ref{l:ep-reg}
and Lemma \ref{l:higher-reg}, we can conclude that for any $0<\delta<T_n$ it holds
\begin{equation}\label{e:uniform-est}
\sup_n\max_{(x,t)\in\Sigma\times [\delta_n, T_n)} \Big(\big|\nabla^{l-1}F_{A_n}\big|+\big|\nabla^l_{A_n}\phi_n\big|\Big)
\le C(l,R_1,\epsilon_0,\delta), \ \forall\ l\ \in\mathbb Z^+.
\end{equation}
Applying (\ref{e:uniform-est}) (with $\delta=\frac{T_n}2$) and taking $l$-order derivatives $\nabla_A^l$ to the equation (\ref{e:heat})
for $l\ge 1$, it is not hard to see that for any $l\in \mathbb Z^+$, $\big(A_n(t), \phi_n(t)\big)$ is uniformly bounded
in $C^l(\Sigma)$ for $0\le t<T_n$. Therefore
there exists $(A_n(T_n), \phi_n(T_n))\in \A\times\S$ such that
$$\lim_{t\uparrow T_n}\big(A_n(t),\phi_n(t)\big)=\big(A_n(T_n), \phi_n(T_n)\big)
\ {\rm{in}}\ C^l(\Sigma), \ \forall \ l\in \mathbb Z^+.$$
This contradicts the maximality of $T_n$. Thus we must have that
$T_n\ge \theta_0 R_1^2$ for all $n\ge 1$.

\smallskip
\noindent{\it Step 3: Weak convergence of $(A_n,\phi_n)$ in $\Sigma\times [0,\theta_0R_1^2]$.}
From the energy inequality (\ref{e:energy-inequ}), we have
  \begin{equation}\label{e:t-energy-bd}
 \sup_{0\le t\le\theta_0R_1^2}\E_n(t)+ \int_0^{\theta_0R_1^2}\int_\Sigma \Big(\big|\frac{\partial\phi_n}{\partial t}\big|^2+\big|{\frac{\partial A_n}{\partial t}}\big|^2\Big)\,dv_gdt
   \le \E_n(0) \le C.
   \end{equation}
This and direct calculations imply that $\big\|A_n(t)\big\|_{L^2(\Sigma)}\in H^1\big([0,\theta_0R_1^2]\big)
\subset C^\frac12\big([0,\theta_0R_1^2]\big)$ and satisfies the estimate: for any $0\le t_1\le t_2\le\theta_0R_1^2$,
\begin{equation} \label{e:holder-est}
\Big| \big\|A_n(t_1)\big\|_{L^2(\Sigma)}-\big\|A_n(t_2)\big\|_{L^2(\Sigma)}\Big|\le \Big\|\frac{\partial A_n}{\partial t}\Big\|_{L^2(\Sigma\times
[0,\theta_0R_1^2])} |t_1-t_2|^\frac12\le C|t_1-t_2|^\frac12.
 \end{equation}
By (\ref{e10}), (\ref{e:holder-est}), and (\ref{e:t-energy-bd}), we also have that $\phi_n$ is bounded in $L^2([0,\theta_0R_1^2], \S_{1,2})$,
i.e.,
\begin{equation} \label{e:spat-energy-bd}
 \sup_{n}\big\|\phi_n\big\|_{L^2([0,\theta_0R_1^2], H^1(\Sigma))}\le C\E_n(0)\le C.
 \end{equation}
 It follows from (\ref{e:t-energy-bd}), (\ref{e:holder-est}), and (\ref{e:spat-energy-bd}) that we may assume that
 there exist a connection $A\in L^\infty([0, \theta_0R_1^2], \A_{0,2})$ and $\phi\in L^2([0,\theta_0R_1^2], \S_{1,2})$
 such that after passing to a subsequence,
 \begin{equation}\label{weak1}
 A_n\rightharpoonup A  \ {\rm{weak^*\ in}}\ L^\infty(L^2); \ \phi_n\rightharpoonup \phi \ {\rm{in}}\ L^2(H^1) \ {\rm{and}}\
 \phi_n\rightarrow \phi \ {\rm{in}}\  L^2(L^2).
 \end{equation}
 From (\ref{e10}), we may assume that there exist $F\in L^\infty([0,\theta_0R_1^2], L^2(\Sigma))$,
 $\psi\in L^\infty([0,\theta_0R_1^2], L^2(\Sigma))$ such that after passing to a subsequence,
 \begin{equation}\label{weak2}
 F_{A_n}\rightharpoonup F; \ D_{A_n}\phi_n\rightharpoonup \psi\  \ {\rm{in}}\ L^2(L^2).
 \end{equation}

\smallskip
\noindent{\it Step 4: Uhlenbeck's gauge fixings and smooth convergence of $(A_n,\phi_n)$.}
In order to show $(A,\phi)$ is a weak solution of (\ref{e:heat}) and (\ref{IVP}), we need to identify $F$ and $\psi$
with $F_A$ and $D_A\phi$ respectively.  To achieve this, we need to control the connections
$A_n$ through Uhlenbeck's gauge fixing techniques. To do it,  we adopt the argument by \cite{HTY} Lemma 3.7.
First note that $(A_n, \phi_n)$ satisfies the smallness condition:
\begin{equation}\label{e:uniform-small}
\max_{0\le t\le \theta_0R_1^2}\E_n(t, R_1)\le\epsilon_0.
\end{equation}
Hence, by Lemma \ref{l:ep-reg} and Lemma \ref{l:higher-reg}, there exists $C>0$ such that for any
$n$, $(A_n, \phi_n)$ satisfies,
for any $k\ge 0$,
\begin{equation}\label{e:higher-est1}
\big\|\nabla^k_{A_n}F_{A_n}(t)\big\|_{L^\infty(\Sigma)}\le \frac{C\epsilon_0}{t^k},
\  \big\|\nabla^{k+1}_{A_n}\phi_n(t)\big\|_{L^\infty(\Sigma)}\le \frac{C\epsilon_0}{t^{\frac{k}2}},
\ \forall\ 0<t\le\theta_0R_1^2.
\end{equation}
Set $t_0=\theta_0R_1^2$. Using (\ref{e:higher-est1}) with $k=0$, we can apply Uhlenbeck's gauge fixing (see \cite{Uh}
Theorem 1.5) to obtain a sequence of Coulomb gauge transforms $s_n\in \G_{1,2}$ such that
$\widetilde{A}_n(t_0)=s_n^*(A_n)(t_0)\in \A$  and satisfies
\begin{equation}\label{slice-uniform-est}
\Big\|\widetilde{A}_n(t_0)\Big\|_{C^k(\Sigma)}\le C\Big(\|A_n(t_0)\|_{L^2(\Sigma)}+\sum_{l=0}^{k-1}\big\|\nabla^l_{A_n}F_{A_n}(t_0)\big\|_{L^\infty(\Sigma)}\Big)\le C(k, t_0, \epsilon_0),
\ \forall\  k\ge 1.
\end{equation}
Set $\widetilde{\phi}_n=s_n^*(\phi_n)$. Since (\ref{e:heat}) is invariant under time-independent gauge transforms,
$(\widetilde{A}_n,\widetilde{\phi}_n)$ is also a solution of (\ref{e:heat}) in
 $[0, t_0]$. Set
 $$\widetilde{\mathcal E}_n(t, R_1):=\max_{x\in\Sigma}\int_{B_{R_1}(x)} e(\widetilde{A}_n(t), \widetilde{\phi}_n(t))\,dv_g.$$
 Then, since the local Yang-Mills-Higgs energy $\E_n(t,R_1)$ is invariant under gauge transformations,
 (\ref{e:uniform-small}) holds for $(\widetilde{A}_n,\widetilde{\phi}_n)$, i.e.,
 $$
\max_{0\le t\le \theta_0R_1^2}\widetilde{\mathcal E}_n(t, R_1)=
\max_{0\le t\le \theta_0R_1^2}\E_n(t, R_1)\le\epsilon_0.
 $$
 so that by Lemma \ref{l:ep-reg}
 and Lemma \ref{l:higher-reg},  (\ref{e:higher-est1}) holds for $(\widetilde{A}_n,\widetilde{\phi}_n)$, i.e.,
 \begin{equation}\label{e:higher-est2}
\big\|\nabla^k_{\widetilde{A}_n}F_{\widetilde{A}_n}(t)\big\|_{L^\infty(\Sigma)}\le \frac{C\epsilon_0}{t^k},
\  \big\|\nabla^{k+1}_{\widetilde{A}_n}\widetilde{\phi}_n(t)\big\|_{L^\infty(\Sigma)}\le \frac{C\epsilon_0}{t^{\frac{k}2}},
\ \forall\ 0<t\le\theta_0R_1^2.
\end{equation}
 Now we take $\partial_t^k\nabla_{\widetilde{A}_n}^l$ of both sides of the equation (\ref{e:heat}) for
 $(\widetilde{A}_n, \widetilde{\phi}_n)$ for any $k,l\ge 1$, and apply (\ref{slice-uniform-est}) and
 (\ref{e:higher-est2}) to conclude that
 $(\widetilde{A}_n, \widetilde{\phi}_n)\in C^\infty(\Sigma\times (0,t_0])$ satisfies
 \begin{equation}\label{e:uniform-est2}
\sup_{n} \Big\|\big(\widetilde{A}_n, \widetilde{\phi}_n\big)\Big\|_{C^k(\Sigma\times [\delta, t_0])}
 \le C(k, \delta, t_0,\epsilon_0), \ \forall\ k\ge 0, \ \forall\ 0<\delta<t_0.
 \end{equation}
From (\ref{e:uniform-est2}), we may assume there exists $\big(\widetilde{A}, \widetilde{\phi}\big)\in C^\infty((0, t_0], \A\times \S)$ such
that  after passing to a subsequence, $(\widetilde{A}_n, \widetilde{\phi}_n)\rightarrow (\widetilde{A},\widetilde{\phi})$ in
$C^k(\Sigma\times [\delta, t_0])$ for any $k\ge 1$ and $0<\delta<t_0$. Since $\big(\widetilde{A}_n,\widetilde{\phi}_n\big)$
are smooth solutions of (\ref{e:heat}) on $\Sigma\times (0,t_0]$, it follows that $\big(\widetilde{A},\widetilde{\phi}\big)$ is also a smooth
solution of (\ref{e:heat}).

Since $\displaystyle\widetilde{A}_n(t_0)=s_n^*(A_n)(t_0)=s_n^{-1} ds_n+s_n^{-1} A_n(t_0)s_n$,  it follows from (\ref{weak1}) and
(\ref{slice-uniform-est}) that
\begin{equation}\label{h1-bound}
\|ds_n\|_{L^2(\Sigma)}\le C\big(\|\widetilde{A}_n(t_0)\|_{L^2(\Sigma)}+\|A_n(t_0)\|_{L^2(\Sigma)}\big)\le C.
\end{equation}
Thus we may assume there exists $s\in \G_{1,2}$ such that after passing to a subquence,
$s_n\rightharpoonup s \ {\rm{in}}\ H^1(\Sigma)$ and $s_n\rightarrow s$ in $L^2(\Sigma)$.
This, combined with (\ref{weak1}), implies that $\big(\widetilde{A},\widetilde{\phi}\big)=\big(s^*A, s^*\phi\big)$ or
equivalently $(A,\phi)=\Big((s^{-1})^*\widetilde{A}, (s^{-1})^*\widetilde\phi\Big)$ in $\Sigma\times (0,t_0]$. Since
(\ref{e:heat}) is invariant under time-independent gauge transformations, we conclude that $(F,\psi)=(F_A, D_A\phi)$,
and $(A,\phi)$ is a weak
solution of (\ref{e:heat})  in $\Sigma\times [0,t_0]$ and satisfies (\ref{IVP}).

\smallskip
\noindent{\it Step 5: Characterization of a finite maximal time interval $T_0$}.
Let $\epsilon_0>0$ be given by Lemma \ref{l:ep-reg}.
First we claim if $0<T_0<+\infty$ is the maximal time interval for a weak solution $(A,\phi)$ constructed through step 1 to step 4, then
 \begin{equation}\label{local-ymh-concen1}
  \limsup_{t\uparrow T_0} \max_{x\in\Sigma}\int_{B_R(x)} e(A(t), \phi(t))\,dv_g\ge \frac{\epsilon_0}2,\ \forall\ R>0.
  \end{equation}
For, otherwise, there exists $R_0>0$ such that for any sufficiently small $\delta>0$,
$$\E(T_0-\delta^2, 2R_0):=\max_{x\in\Sigma}\int_{B_{2R_0}(x)\times \{T_0-\delta^2\}} e(A(t), \phi(t))\,dv_g<\frac{\epsilon_0}2.$$
This, combined with lemma \ref{l:energy-inequ-local}, implies that there exists $\theta_0\in (0,1)$ such that for $T_0-\delta^2\le t\le T_0-\delta^2+(\theta_0 R_0)^2$,
$$\E(t, R_0)\le \E(T_0-\delta^2, 2R_0)+\frac{C(t-(T_0-\delta^2))}{R_0^2}\E(0)\le \frac{\epsilon_0}2+ C\theta_0^2 \E(0)\le \epsilon_0.$$
By choosing $\delta=\theta_0R_0$, this implies that
\begin{equation}\label{nonsing-time}
\sup_{T_0-\theta_0^2R_0^2\le t\le T_0}\E(t,R_0)\le\epsilon_0.
\end{equation}
From (\ref{nonsing-time}), we can repeat the argument from step 1 to step 4 to conclude that there exists a $s\in\G_{1,2}$
such that $\big(s^*A(T_0), s^*\phi(T_0)\big)\in \A\times \S$ and hence $T_0$ is not a maximal time interval. This contradicts the definition of $T_0$.
The improvement of $\frac{\epsilon_0}2$ in (\ref{local-ymh-concen1}) to $\alpha(M)$ in (\ref{local-ymh-concen}) follows from the blow-up analysis
performed near $T_0$, see section 6 below.
\end{proof}

For the uniqueness of weak solutions, we have

\begin{thm}
  If the initial data belongs to $W^{2,p}$ for $p>2$, then the weak solution is unique.
\end{thm}
\begin{proof}
  The proof follows exactly as the smooth case in Theorem~\ref{t:smooth}. Note that when $A\in W^{2,p}, p>2$, the coefficients of equation~(\ref{e:gauge-DeTurck}) belong to $C^0$ and the uniqueness of solutions to the ordinary differential equation
is guaranteed. Thus the arguments in the proof of uniqueness part of Theorem~\ref{t:smooth} holds the same in this case.
\end{proof}

\section{Compactness of approximate YMH fields}

In \cite{So1}, the first author discussed in detail the convergence and blow-up behavior of a sequence of YMH fields. It was shown that a sequence of YMH fields with bounded YMH energy converges to a YMH field along with
possibly finitely many bubbles, i.e., nontrivial harmonic maps from $\mathbb S^2$, which are attached to the limiting YMH field. This phenomenon, referred as {\it bubble tree} convergence,  has first been established in the study of compactness of harmonic maps from Riemann surfaces, see for example \cite{P1}.

Extending the arguments in \cite{So1}, we can prove the bubble tree convergence of a sequence of approximate YMH fields, which is needed to describe the asymptotic behavior of the YMH heat flow at both finite singular times and the time infinity. More precisely, for a pair $(A, \phi)\in \A_{1,2}\times \S_{1,2}$, set
\[ \tau_1(A,\phi) = D_A^*F_A+\phi^*D_A\phi;\  ~ \tau_2(A, \phi) = D^*_AD_A\phi+(\mu(\phi)-c)\cdot\nabla\mu(\phi). \]
We will show the bubble tree convergence of a sequence $(A_n, \phi_n)\in \A_{1,2}\times \S_{1,2}$, with bounded YMH energies $\E(A_n,\phi_n)$,  satisfying
\begin{equation}\label{e:bounded-tension}
\norm{\tau_1(A_n,\phi_n)}_{L^2(\Sigma)} + \norm{\tau_2(A_n, \phi_n)}_{L^2(\Sigma)} \le C.
\end{equation}

\begin{thm}\label{t:compact}
Suppose $(A_n, \phi_n)\in \A_{1,2}\times \S_{1,2}$ is a sequence of fields, with bounded YMH energies $\E(A_n,\phi_n)$, which satisfies~(\ref{e:bounded-tension}).
Then there exist a subsequence of $(A_n,\phi_n)$, still denoted as $(A_n, \phi_n)$, a set of finitely many points $\mathbf{x} = \{x_1, x_2,\cdots, x_k\}\subset \Si$,
 and an approximate YMH field $(A_\infty, \phi_\infty)\in \A_{2,2}\times \S_{2,2}$, with $L^2$-tension fields
 $\tau_1(A_\infty, \phi_\infty)$ and $\tau_2(A_\infty, \phi_\infty)$, such that the following properties hold:
 \begin{enumerate}
 \item There exist gauge transformations $\{s_n\}\subset\G_{1,2}$ such that $s_n^*A_n\rightarrow A_\infty$ in $H^1(\Sigma)$ and
$s_n^*\phi_n\rightarrow\phi_\infty$ in $H^1_{loc}(\Si\setminus\mathbf{x})$.
\item There exist finitely many nontrivial harmonic maps $\omega_{ij}:\mathbb S^2\to M, 1\le i\le k, 1\le j\le l$,  such that
\[\lim_{n\to \infty} \E(A_n, \phi_n) = \E(A_\infty, \phi_\infty) + \sum_{1\le i\le k,1\le j\le l}\E(\omega_{ij}),\]
where $\displaystyle \E(\omega_{ij}) = \int_{\mathbb S^2}\big|\nabla\omega_{ij}\big|^2dv_{g_0}$ is the Dirichlet energy of $\omega_{ij}$.
\item The images of the bubbles $\displaystyle \{\omega_{ij}\}_{1\le i\le k, 1\le j\le l}$ and the limiting map $\phi_\infty$ are connected.
\end{enumerate}
\end{thm}

Theorem \ref{t:compact}  has been proved by \cite{So1} for YMH fields $(A_n,\phi_n)$, i.e,
$\tau_1(A_n,\phi_n)=\tau_2(A_n,\phi_n)=0$.  It turns out the argument in \cite{So1} can be modified
to prove Theorem \ref{t:compact} on approximate YMH fields satisfying the condition (\ref{e:bounded-tension}).
Here we sketch a proof. First we need the following lemma.

\begin{lem}\label{l:reg} There exists $\epsilon_0>0$ such that for $x_0\in\Sigma$ and $r_0>0$ if $A \in \A_{1,2}\big |_{B_r(x_0)}$
is a connection on the ball $ B_{r_0}(x_0)\subset\Sigma$, with $D_A^*F_A\in L^2(B_{r_0}(x_0))$, which satisfies
\begin{equation}\label{small-curvature}
r_0\big\|F_A\big\|_{L^2(B_{r_0}(x_0))}\le\epsilon_0,
\end{equation}
then there exists a gauge transformation $s\in \G_{2,2}$ such that the following estimate holds:
\begin{equation}\label{h2-estimate}
  \norm{s^*A}_{H^2(B_{\frac{r_0}2}(x_0))} \le C(r_0)(\norm{F_A}_{L^2(B_{r_0}(x_0))} + \norm{D_A^*F_A}_{L^2(B_{r_0}(x_0))}).
\end{equation}
\end{lem}
\begin{proof} If we define $\widehat{A}(x)=r_0A(x_0+r_0x)$ and $\widehat{g}(x)=g(x_0+r_0x)$ for $x\in B_1$, then
by simple scaling arguments we have that $\widehat{A} \in \A_{1,2} \big|_{B_1}$, with $D_{\widehat A}^*F_{\widehat{A}}\in L^2(B_1)$,
and satisfies
$$\big\|D_{\widehat A}^*F_{\widehat{A}}\big\|_{L^2(B_1)}
=r_0^2 \big\|D_{A}^*F_A\big\|_{L^2(B_{r_0}(x_0)}\le \big\|D_{A}^*F_A\big\|_{L^2(B_{r_0}(x_0)},$$
and
$$\int_{B_1} |F_{\widehat{A}}|^2\,dv_{g_0}
=\int_{B_{r_0}(x_0)} r_0^2|F_{A}|^2\,dv_g<\epsilon_0^2.$$
Thus we may assume, for simplicity, that $x_0=0$ and $r_0=1$.
Since $A\in\A_{1,2}$ satisfies (\ref{small-curvature}), it follows from \cite{Uh} Theorem 1.3
that there exists a gauge transform $s\in \G_{1,2}$ such that $\widetilde A=s^*A$ satisfies
\begin{equation}\label{coulomb-gauge1}
d^*\widetilde A=0, \ \big\|\widetilde A\big\|_{H^1(B_1)}\le C\big\|F_A\big\|_{L^2(B_1)}.
\end{equation}
Using $d^*\widetilde A=0$, we obtain the following elliptic equation:
\begin{equation}\label{e:ep1}
\Delta \widetilde A + \big[\widetilde A, d\widetilde A\big] + \big[\widetilde A, [\widetilde A,\widetilde A]\big]
=D_{\widetilde A}^*F_{\widetilde A}.
\end{equation}
By Sobolev's embedding $H^1(B_1) \hookrightarrow L^q(B_1)$ for any $1<q<+\infty$, and the standard $W^{2,p}$-estimate
on (\ref{e:ep1}), we have that  for any $1<p<2$, $\widetilde A\in W^{2,p}(B_{\frac34})$ and
$$\big\|\widetilde A\big\|_{W^{2,p}(B_{\frac34})}\le C\Big(\big\|\widetilde A\big\|_{H^1(B_1)}
+\big\|D_{\widetilde A}^*F_{\widetilde A}\big\|_{L^2(B_1)}\Big)\le C\big(\norm{F_A}_{L^2(B_1)} + \norm{D_A^*F_A}_{L^2(B_1)}\big).
$$
This, combined with Sobolev's embedding $W^{2,p}(B_{\frac34}) \hookrightarrow C^0(B_{\frac34})$,
yields $\widetilde A\in C^0(B_{\frac34})$ and
$$
\big\|\widetilde A\big\|_{L^\infty(B_{\frac34})}\le C\big(\norm{F_A}_{L^2(B_1)} + \norm{D_A^*F_A}_{L^2(B_1)}\big).
$$
Now we can apply $W^{2,2}$-estimate to (\ref{e:ep1}) to obtain that $\widetilde A\in W^{2,2}(B_{\frac12})$ satisfies the desired
estimate (\ref{h2-estimate}). This, after scaling back to the original scale, completes the proof.
\end{proof}

Following the gluing procedure given by \cite{Uh} and Lemma \ref{l:reg}, we obtain the following proposition.
\begin{prop} \label{l:global-gauge-fix} For any $\Lambda>0$, there exists $C(\Lambda)>0$ such that
if $A\in \A_{1,2}$, with $D_A^*F_A\in L^2(\Sigma)$, has
\begin{equation}
\|F_A\|_{L^2(\Sigma)}\le \Lambda,
\end{equation}
then there is a gauge transform $s\in \G_{2,2}$
such that $\widetilde A=s^*A\in \A_{2,2}$ satisfies
\begin{equation}\label{h2-estimate3}
\big\|\widetilde A\big\|_{H^2(\Sigma)}\le C(\Lambda)\Big(\big\|F_A\big\|_{L^2(\Sigma)}+\big\|D_A^*F_A\big\|_{L^2(\Sigma)}\Big).
\end{equation}
\end{prop}
\begin{proof} Choose $0<r_0\le \frac{\epsilon_0}{\Lambda}$, we have
$$r_0\big\|F_A\big\|_{L^2(B_{r_0}(x_0))}\le \epsilon_0$$
holds for any $x_0\in\Sigma$.
Since $\Sigma$ is compact, there exist a positive integer $k_0 \le Cr_0^{-2}$ and points $\{x_1,\cdots, x_{k_0}\}\subset\Sigma$
such that $\Sigma$ is covered by $\{B_{\frac{r_0}2}(x_i)\}_{1\le i\le k_0}$.
Applying Lemma \ref{l:reg} on each $B_{r_0}(x_i)$, $1\le i\le k_0$, and using the gluing technique by \cite{Uh}, we can find a
gauge transform $s\in \G_{1,2}$ such that $\widetilde A=s^*A\in \A_{2,2}$ satisfies, for $1\le i\le k_0$,
\begin{equation}\label{h2-estimate4}
\big\|\widetilde A\big\|_{H^2(B_{\frac{r_0}2}(x_i))}^2\le C(r_0)\Big(\big\|F_A\big\|_{L^2(B_{r_0}(x_i))}^2
+\big\|D_A^*F_A\big\|_{L^2(B_{r_0}(x_i))}^2\Big).
\end{equation}
It is clear that (\ref{h2-estimate3}) follows by summing up (\ref{h2-estimate4}) over $1\le i\le k_0$.
Since $ds=s A-\widetilde A s$, it is easy to see that $s\in \G_{2,2}$. This completes the proof.
\end{proof}

Before presenting the proof of Theorem \ref{t:compact}, we recall the well-known bubble-tree convergence result of a sequence of approximate harmonic maps,
 with bounded Dirichlet energies and $L^2$-tension fields, which was proved in~\cite{DT} and~\cite{LinW}.

\begin{thm}\label{t:bubble-convergence-harmonic}
Let $\{u_n\}\subset H^2(B_1,M)$ have uniformly bounded Dirichlet energies $\E(u_n, B_1)$ and have their tension fields
$\tau(u_n):=\Delta_g u_n+\Pi(u_n)(du_n, du_n)$ uniformly bounded in $L^2(B_1)$.
Then there exist an approximate harmonic map $u_\infty\in H^2(B_1,M)$, with tension field $\tau(u_\infty)\in L^2(B_1)$,
 such that after passing to a subsequence, $u_n$ bubble tree converges to $u_\infty$.
More precisely, there exist finitely many points $\mathbf{x} = \{x_1, x_2,\cdots, x_k\}\subset B_1$ such that the following properties hold:
\begin{enumerate}
\item $u_n$ converges to $u_\infty$ strongly in $H^1_{loc}(B_1\setminus\mathbf{x})$.
\item There exist finitely many bubbles (i.e. nontrivial harmonic maps) $\omega_{ij}:\mathbb S^2\to M, 1\le i\le k, 1\le j\le l$, such that
\[\lim_{n\to \infty} \E(u_n, B_r) = \E(u_\infty, B_r) + \sum_{1\le i\le k,1\le j\le l}\E(\omega_{ij}),\]
where $0<r<1$ is such that ${\bf x}\subset B_r$.
\item The images of the bubbles $\displaystyle\{\omega_{ij}\}_{1\le i\le k, 1\le j\le l}$ and the limiting map $u_\infty$ are connected.
\end{enumerate}
\end{thm}

Now we are ready to prove  Theorem \ref{t:compact}.

\begin{proof}[Proof of Theorem~\ref{t:compact}]
Since $A_n$ has uniformly bounded $L^2$-curvatures,  Proposition \ref{l:global-gauge-fix} implies
that there exists
a sequence of gauge transformations $\{s_n\}\subset \G_{2,2}$ such that $\widetilde{A}_n:=s_n^* A_n\in \A_{2,2}$ satisfies
\begin{eqnarray}\label{h2-estimate1}
\big\|\widetilde{A}_n\big\|_{H^2(\Sigma)}&\le& C\Big(\big\|F_{A_n}\big\|_{L^2(\Sigma)}+\big\|D_{A_n}^*F_{A_n}\big\|_{L^2(\Sigma)}\Big)\nonumber\\
&\le& C\big(\|F_{A_n}\|_{L^2(\Sigma)}+\|\tau_1(A_n,\phi_n)\|_{L^2(\Sigma)}+\|D_{A_n}\phi_n\|_{L^2(\Sigma)}\big)\le C.
\end{eqnarray}
Therefore there exists $\widetilde A\in \A_{2,2}$ such that $\widetilde A_n\rightarrow \widetilde A$ strongly in $H^1(\Sigma)$.
Set $\widetilde{\phi}_n=s_n^*(\phi_n)$. It follows that $\widetilde\phi_n\in \S_{2,2}$. Now we claim that on each local chart
$U$ of $\Sigma$,  $\widetilde\phi_n$ can be viewed as  a sequence of approximate harmonic maps from $(U, g)$ to $M$, with
uniformly bounded $L^2$-tension fields $\tau(\widetilde\phi_n)$. In fact, since $\F\big|_U$ is a trivial bundle, we can
identify $\widetilde \phi_n$ on $U$ as a map $u_n$  from $U$ to $M$.
To write the equation of $u_n$, we isometrically embed $M$ into some Euclidean space $\mathbb R^L$ and let
$\Pi$ denote its second fundamental form. Write $D_{\widetilde{A}_n}=d+\widetilde{A}_n$ and
$D_{\widetilde{A}_n}^*=d^*+\widetilde{A}_n^*$. It is easy to see that $\widetilde{A}_n^*=-\widetilde{A}_n$. Hence we
have, on $U$,
\begin{eqnarray}\label{e:u-equation}
&&\tau(u_n):=\Delta_g u_n+\Pi(u_n)(du_n, du_n)\nonumber\\
&&=\tau_2(\widetilde{A}_n, \widetilde{\phi}_n)-(\mu(\widetilde\phi_n)-c)\nabla\mu(\widetilde\phi_n)
-d^*\widetilde{A}_n \cdot u_n-2\widetilde{A}_n\cdot du_n -\widetilde{A}_n^2\cdot u_n.
\end{eqnarray}
Here $\Delta_g$ is the Laplace operator with respect to the metric $g$ in $U$.
Applying (\ref{h2-estimate1}) to (\ref{e:u-equation}), we obtain that
\begin{eqnarray*}
&&\norm{\tau(u_n)}_{L^2(U)} \\
&&\le C\Big[\norm{\tau_1(A_n,\phi_n)}_{L^2(U)} + \norm{\tau_2(\widetilde{A}_n, \widetilde{\phi}_n)}_{L^2(U)}
+ \norm{F_{A_n}}_{L^2(U)} + \norm{D_{\widetilde A_n} \widetilde{\phi_n}}_{L^2(U)}\Big] \\
&&\le C\Big[\norm{\tau_1(A_n,\phi_n)}_{L^2(U)} + \norm{\tau_2({A}_n, {\phi}_n)}_{L^2(U)}
+ \norm{F_{A_n}}_{L^2(U)} + \norm{D_{A_n} {\phi_n}}_{L^2(U)}\Big]\\
&&\le C.
\end{eqnarray*}
This implies that $\tau(u_n)$ is uniformly bounded in $L^2(U)$. Thus we may apply Theorem~\ref{t:bubble-convergence-harmonic} on each local chart
$U$ to conclude that $\widetilde\phi_n$ {\it bubble-tree} converges to a limit section $\phi_\infty\in \S_{2,2}$, namely
the properties 2 and 3 in Theorem 6.1 hold.  This completes the proof.
\end{proof}

\section{Asymptotic behavior of YMH flow at finite time singularities and time infinity}

In this section, we will discuss the asymptotic behaviors of the global weak solution $(A,\phi)$ to the YHM  flow equation
(\ref{e:heat}) under the initial condition (\ref{IVP}), established in Theorem 5.1, at each possible finite singular time $T_i$, $1\le i\le L$, and at time
infinity. The main results of this section are consequences of the compactness properties on
approximate YMH fields obtained in the previous section.

\begin{thm}\label{t:finite-sing} For $(A_0,\phi_0)\in \A_{1,2}\times \S_{1,2}$, assume that $0<T_1<+\infty$ is the maximal time interval for the
local weak solution $(A, \phi)$ to the YMH  flow equation (\ref{e:heat}), under the initial condition (\ref{IVP}), constructed
by Theorem 5.1. Then the following properties hold:
\begin{enumerate}
\item There exists a pair $(A(T),\phi(T))\in \A_{0,2}\times\S_{1,2}$
such that $(A(t),\phi(t))\rightarrow (A(T),\phi(T))$ in $L^2(\Sigma)$, and $\phi(t)\rightharpoonup\phi(T)$
 in $H^1(\Sigma)$, as $t\rightarrow T^-$.
\item There exist a set of finitely many points $\mathbf{x} := \{x_1, \cdots, x_k\}\subset \Sigma$, with $k\le \frac{\L(0)}{\epsilon_0}$,
 and a gauge transformation $s\in \G_{1,2}$ such that $(s^*A(t), s^*\phi(t))$ converges to $(s^*A(T), s^*\phi(T))$
 in $C^\infty_{loc}(\Si\setminus\mathbf{x})$, as $t\rightarrow T^-$.
\item There exist finitely many nontrivial harmonic maps $\omega_{ij}: \mathbb S^2 \to M, 1\le i \le k, 1\le j\le l$ such that
\begin{equation}
  \lim_{t\to T^-} \E(A(t), \phi(t)) = \E(A(T), \phi(T)) + \sum_{1\le i\le k, 1\le j\le l}\E(\omega_{ij}).
\end{equation}
\end{enumerate}
\end{thm}

\begin{proof}  We may assume, after a suitable gauge fixing, that
$(A,\phi)\in C^\infty(\Sigma\times (0,T))$.  Since $(A,\phi)\in C([0,T], L^2(\Sigma))$ and $\phi\in L^\infty([0,T], H^1(\Sigma))$, the property 1
follows easily.
For the property 2,  define the energy concentration set  ${\bf x}\subset \Sigma$ by
\begin{equation}\label{sing_set}
{\bf x}:=\bigcap_{r>0}\Big\{x\in\Sigma: \ \limsup_{t\uparrow T^{-}}\int_{B_r(x)} \big(|F_A|^2+|D_A\phi|^2+|\mu(\phi)-c|^2\big)(t)\,dv_g\ge \epsilon_0
\Big\}.
\end{equation}
By (\ref{e:energy-inequ}) and a simple covering argument,
we can show that ${\bf x}=\{x_1,\cdots, x_L\}\subset\Sigma$ is a set of finitely many points, with $L\le \frac{\mathcal E(0)}{\epsilon_0}$.
For any
$x_0\in \Sigma\setminus {\bf x}$, it follows from (\ref{sing_set}) that there exist $r_0>0$ and $\delta_0>0$ such that
$$
\int_{B_{r_0}(x_0)} \big(|F_A|^2+|D_A\phi|^2+|\mu(\phi)-c|^2\big)(t)\,dv_g <\epsilon_0, \ \forall\  T-\delta_0^2\le t<T.
$$
Hence, for any compact set $K\subset \Sigma\setminus {\bf x}$, by a simple covering argument
there exists $r_1=r_1(K)>0$ such that
\begin{equation}\label{uniform-small}
\sup_{x\in K} \int_{B_{r_1}(x)} \big(|F_A|^2+|D_A\phi|^2+|\mu(\phi)-c|^2\big)(t)\,dv_g <\epsilon_0, \ \forall\ T-r_1^2\le t<T.
\end{equation}
Hence by Lemma \ref{l:ep-reg} and Lemma \ref{l:higher-reg} we have that
$$
\max_{x\in K}\Big\|\big|\nabla^l_A F_A\big|+\big|\nabla_A^{l+1}\phi\big|\Big\|_{L^\infty\big(B_{\frac{r_1}2}(x)\times [T-\frac{r_1^2}4, T)\big)}
\le C(l, \epsilon_0, K), \ \forall\ l\ge 0.
$$
Similar to the argument in the proof of Theorem 5.1, this, combined with the YMH  flow equation (\ref{e:heat}) and the fact $A(t)\in C^\infty(\Sigma)$ for $t=T-\frac{r_1^2}4$, implies that  $(A,\phi)\in C^\infty\big(K_{\frac{r_1}2}\times [T-\frac{r_1^2}4, T]\big)$ and
\begin{equation}\label{higher-reg1}
\big\|(A,\phi)\big\|_{C^l\big(K_{\frac{r_1}2}\times [T-\frac{r_1^2}4, T]\big)}\le C(l,\epsilon_0, K), \ \forall\ l\ge 0.
\end{equation}
Here $K_{r}:=\big\{x\in\Sigma: {\rm{dist}}(x,K)< r\big\}$ denotes $r$-neighborhood of $K$, for $r>0$.
It follows from (\ref{higher-reg1}) and the property 1 that $(A(t), \phi(t))\rightarrow (A(T),\phi(T))$ in $C^\infty(K_{\frac{r_1}4})$
as $t\rightarrow T^-$. This yields property 2, since $K\subset \Sigma\setminus {\bf x}$ is arbitrary.

For the property 3,  we may assume, for simplicity, that ${\bf x}=\{x_0\}$ only consists of a single point.
First we need to show\\
\noindent{\it Claim 1}. There exists a constant $m\ge \epsilon_0$ such that
\begin{equation}\label{radon1}
e(A(t),\phi(t))\,dv_g\rightharpoonup e(A(T),\phi(T))\,dv_g +m \delta_{x_0}
\end{equation}
for $t\rightarrow T^{-}$, as convergence of Radon measures on $\Sigma$. Here as in the section 1,
$$e(A(t),\phi(t))=(|F_A|^2+|\nabla_A\phi|^2+|\mu(\phi)-c|^2)(t),$$
and $\delta_{x_0}$ denotes the delta mass centered at $x_0$.

The proof is based on Lemma \ref{l:energy-inequ-local}
and can be carried out similarly to \cite{LinW} Lemma 4.1. For the convenience of readers, we sketch it here.
For ${\bf x}=\{x_0\}$ and  $(A(t),\phi(t))\rightarrow (A(T),\phi(T))$ in
$H^1_{\rm{loc}}(\Sigma\setminus\{x_0\})$ by the property 2, we have that $F_{A(t)}\rightharpoonup F_{A(T)}$
in $L^2(\Sigma)$. Hence, for given two sequences $t_i^{j} \rightarrow T^{-}$, there exist
$m_j\ge \epsilon_0$, $j=1,2$, such that
\[
e(A(t_i^{j}),\phi(t_i^j))\,dv_g\rightharpoonup e(A(T),\phi(T))\,dv_g +m_j \delta_{x_0}, \ {\rm{for}}\ i\rightarrow\infty,
\]
as convergence of Radon measures, for $j=1,2$. It suffices to show $m_1=m_2$. For any $\epsilon>0$, choose
$\delta>0$ such that $\displaystyle\int_{B_{2\delta}(x_0)}e(A(T),\phi(T))\,dv_g<\epsilon$. Thus we have
\begin{eqnarray*}
&&m_1\ge \int_{B_{2\delta}(x_0)}e(A(t_i^{1}),\phi(t_i^1))\,dv_g-\epsilon\\
&&\ge \int_{B_{\delta}(x_0)}e(A(t_i^{2}),\phi(t_i^2))\,dv_g-C\frac{|t_i^1-t_i^2|}{\delta^2}\mathcal E(0)
-\Big|\int_{t_i^1}^{t_i^2}\int_\Sigma\big(|\frac{\partial A}{\partial t}|^2+|\frac{\partial \phi}{\partial t}|^2\big)\Big|-\epsilon\\
&&\ge  \int_{B_{\delta}(x_0)}e(A(t_i^{2}),\phi(t_i^2))\,dv_g-2\epsilon\ge m_2-2\epsilon.
\end{eqnarray*}
This yields $m_1\ge m_2$. Similarly, we have $m_2\ge m_1.$ Hence $m_1=m_2$ and (\ref{radon1}) follows.

To simplify the presentation, assume further that $(\Sigma,g)=(\mathbb R^2, dx^2)$, $x_0=(0,0)\in\mathbb R^2$, and $T=0$.
From (\ref{radon1}), there exist $t_i\uparrow 0$ and $\lambda_i\downarrow 0$ such
that
\begin{equation}
\lim_{t_i\uparrow 0}\int_{B_{\lambda_i}}e(A,\phi)(x,t_i)\,dx=m.
\end{equation}
Define $A_i(x,t)=\lambda_i A(\lambda_i x, t_i+\lambda_i^2 t), \phi_i(x,t)=\phi(\lambda_i x, t_i+\lambda_i^2 t)$
for $(x,t)\in \mathbb R^2\times [-2, 0]$. Then $(A_i, \phi_i)$ solves the scaled version of
the YMH heat flow equation (\ref{e:heat}):
\begin{equation}\label{e:heat1}
\begin{cases}\displaystyle\frac{\partial A_i}{\partial t}=-D_{A_i}^*F_{A_i}-\lambda_i^2\phi_i^*D_{A_i}\phi_i,\\
\displaystyle\frac{\partial \phi_i}{\partial t}=-D_{A_i}^*D_{A_i}\phi_i-\lambda_i(\mu(\phi_i)-c)\cdot \nabla (\mu(\phi_i)),
\end{cases}
\end{equation}
and satisfies
\begin{equation}
\int_{-2}^0\int_{\mathbb R^2}
\Big(\lambda_i^{-2}\big|\frac{\partial A_i}{\partial t}\big|^2+\big|\frac{\partial \phi_i}{\partial t}\big|^2\Big)
=\int_{t_i-2\lambda_i^2}^{t_i}\int_{\mathbb R^2}
\Big(\big|\frac{\partial A}{\partial t}\big|^2+\big|\frac{\partial \phi}{\partial t}\big|^2\Big)\rightarrow 0,
\ {\rm{as}}\ i\rightarrow\infty.
\end{equation}
Therefore, by Fubini's theorem, there exists $\tau_i\in (-1,-\frac12)$ such that
\begin{equation}\label{good-slice}
\lim_{i\rightarrow\infty}\int_{\mathbb R^2}
\Big(\lambda_i^{-2}\big|\frac{\partial A_i}{\partial t}\big|^2+\big|\frac{\partial \phi_i}{\partial t}\big|^2\Big)(x,\tau_i)
=0.
\end{equation}
From the local energy inequality (\ref{e:energy-inequ-local}), we also have
$$
\int_{B_{R\lambda_i}}e(A,\phi)(x,t_i+\lambda_i^2\tau_i)
\ge \int_{B_{\lambda_i}}e(A,\phi)(x,t_i)-CR^{-2}\mathcal E(0)\ge m+o(1)-CR^{-2}\mathcal E(0).
$$
This and (\ref{radon1}) imply that
\begin{equation}\label{energy-loss1}
\lim_{R\rightarrow +\infty}\lim_{i\rightarrow\infty}\int_{B_{R\lambda_i}}e(A,\phi)(x,t_i+\lambda_i^2\tau_i)=m.
\end{equation}
Define $B_i(x)=A_i(x,\tau_i)$ and $\psi_i(x)=\phi_i(x,\tau_i)$ for $x\in\mathbb R^2$.
It follows from the (\ref{energy-loss1}),
(\ref{good-slice}), and (\ref{e:heat1})  that $(B_i, \psi_i)\in \A\times \S$ satisfies:
\begin{equation}\label{energy-loss2}
\int_{B_R}\big(\lambda_i^{-2}|F_{B_i}|^2+|D_{B_i}\psi_i|^2+\lambda_i^2|\mu(\psi_i)-c|^2\big)(x)
=\int_{B_{R\lambda_i}}e(A,\phi)(x,t_i+\lambda_i^2\tau_i)=m+o(1),
\end{equation}
\begin{equation}\label{l2-bound}
\int_{B_R}|B_i(x)|^2=\int_{B_{R\lambda_i}}|A|^2(x,t_i+\lambda_i^2\tau_i)\le C,
\end{equation}
\begin{equation}\label{small-F}
\int_{B_R}\big|F_{B_i}(x)\big|^2=\lambda_i^2\int_{B_{R\lambda_i}}|F_{A}|^2(x,t_i+\lambda_i^2\tau_i)\le m\lambda_i^2,
\end{equation}
\begin{eqnarray}\label{small-DF}
\int_{B_R}\big|D_{B_i}^*F_{B_i}(x)\big|^2&\le& C\int_{B_R}\big(|\frac{\partial A_i}{\partial t}|^2+\lambda_i^4|D_{A_i}\phi_i|^2\big)(x,t_i+\lambda_i^2\tau_i)\le C(o(1)+\lambda_i^2)\lambda_i^2,
\end{eqnarray}
and
\begin{equation}\label{approx-hm1}
\int_{B_R}\big|D_{B_i}^*D_{B_i}\psi_i(x)\big|^2
\le C\int_{B_R}\big(|\frac{\partial \phi_i}{\partial t}|^2+\lambda_i^2|D_{A_i}\phi_i|^2\big)(x,t_i+\lambda_i^2\tau_i)\le C\lambda_i^2+o(1).
\end{equation}
From (\ref{l2-bound}, (\ref{small-F}), and (\ref{small-DF}), we can apply Lemma \ref{l:reg} to conclude that on $B_R$
there exist a sequence of gauge transformations $\{s_i\}\subset \G_{2,2}$ such that
$\widetilde{B}_i:=s_i^*B_i$ satisfies
\begin{equation}\label{small-H1}
\big\|\widetilde{B}_i\big\|_{H^2(B_R)}\le C(R)\big(\|F_{B_i}\|_{L^2(B_R)}+\|D_{B_i}^*F_{B_i}\|_{L^2(B_R)}\big)\le C(R)\lambda_i^2.
\end{equation}
Thus we may assume that $\widetilde{B}_i$ converges to $0$  weakly in $H^2(B_R)$  and strongly in $H^1(B_R)$. While (\ref{approx-hm1})
implies that $\widetilde{\psi}_i=s_i^*\psi_i$ satisfies
\begin{equation}\label{approx-hm2}
\int_{B_R}\big|D_{\widetilde{B}_i}^*D_{\widetilde{B}_i}\widetilde{\psi}_i(x)\big|^2
=\int_{B_R}\big|D_{B_i}^*D_{B_i}\psi_i(x)\big|^2\le C\lambda_i^2+o(1).
\end{equation}
From (\ref{small-H1}) and (\ref{approx-hm2}),  we can verify, similar to that of Theorem 6.1, that
$\big\{\widetilde{\psi}_i\big\}$ is a sequence of approximate harmonic maps, whose tension fields
$\tau(\widetilde{\psi}_i)$ have $\displaystyle\big\|\tau(\widetilde{\psi}_i)\big\|_{L^2(B_R)}\le C\lambda_i$.
Now we need

\smallskip
\noindent{\it Claim 2}. After passing to a subsequence, $\widetilde{\psi}_i$ converges to a constant map in $L^2(B_R)$.

To see this, observe that
\begin{equation}\label{l2-vanish}
\int_{B_R}\big|\widetilde\psi_i(x)-s_i^*\phi_i(x, -\lambda_i^{-2}t_i)\big|^2=\int_{B_R}\big|\phi_i(x, \tau_i)-\phi_i(x, -\lambda_i^{-2}t_i)\big|^2
\le 4\int_{-2}^2\int_{B_R}\big|\frac{\partial\phi_i}{\partial t}\big|^2\rightarrow 0,
\end{equation}
\begin{equation}\label{small-h1norm1}
\int_{B_R}\big|D_{s_i^*A_i} s_i^*\phi_i\big|^2(x,-\lambda_i^{-2}t_i)
=\int_{B_R}\big|D_{A_i}\phi_i|^2(x,-\lambda_i^{-2}t_i)=\int_{B_{R\lambda_i}}\big|D_{A}\phi\big|^2(x,0)\rightarrow 0,
\end{equation}
and
\begin{eqnarray}&&\int_{B_R}\big|s_i^*A_i(x, -\lambda_i^{-2}t_i)\big|^2
\le 2\int_{B_R}\big|s_i^*A_i(x, -\lambda_i^{-2}t_i)-\widetilde{B}_i(x)\big|^2+2\int_{B_R}\big|\widetilde{B}_i(x)\big|^2\nonumber\\
&&\le 2\int_{B_R}\big|s_i^*A_i(x, -\lambda_i^{-2}t_i)-s_i^*A_i(x,\tau_i)\big|^2+2\int_{B_R}\big|\widetilde{B}_i(x)\big|^2\nonumber\\
&&=2\int_{B_R}\big|A_i(x, -\lambda_i^{-2}t_i)-A_i(x,\tau_i)\big|^2+2\int_{B_R}\big|\widetilde{B}_i(x)\big|^2\nonumber\\
&&\le 4\int_{-2}^2\int_{B_R}\big|\frac{\partial A_i}{\partial t}\big|^2+2\int_{B_R}\big|\widetilde{B}_i(x)\big|^2\rightarrow 0.
\label{small-h1norm2}
\end{eqnarray}
Claim 2 now follows from (\ref{l2-vanish}), (\ref{small-h1norm1}), and (\ref{small-h1norm2}).

For each $R>0$, we now can apply Theorem 6.4 to $\widetilde{\psi}_i$ on $B_R$ to conclude that there exist $N_R$ bubbles
$\big\{\omega_{l,R}\big\}_{l=1}^{N_R}$, with $N_R\le \frac{m}{\alpha(M)}$, such that
\begin{equation}\label{energy-id1}
\lim_{i\rightarrow\infty}\int_{B_R} \big|D_{\widetilde{B}_i}\widetilde{\psi}_i\big|^2
=\sum_{l=1}^{N_R}\mathcal E(\omega_{l,R}).
\end{equation}
We may assume that there exists an integer $d\in \big[1, \frac{m}{\alpha(M)}\big]$ such that for $N_R=d$ for $R>>1$.
Note that for $l=1,\cdots, d$, $\big\{\omega_{l,R}\big\}_{R>1}$ is a sequence of harmonic maps from $\mathbb S^2$ to $M$
whose energies are uniformly bounded. Hence we can apply the energy identity result for harmonic maps (cf. \cite{P1})
to conclude that for $l=1,\cdots, d$, there exist  $N_l$ bubbles $\big\{\omega_{l,j}\big\}_{j=1}^{N_l}$,
with $N_l\le \frac{m}{\alpha(M)}$, such that
\begin{equation}\label{energy-id2}
\lim_{R\rightarrow\infty}\mathcal E(\omega_{l,R})=\sum_{j=1}^{N_l}\mathcal E(\omega_{l,j}).
\end{equation}
Hence we have
\begin{equation}\label{energy-id3}
\lim_{R\rightarrow\infty}\lim_{i\rightarrow\infty}\int_{B_R} \big|D_{\widetilde{B}_i}\widetilde{\psi}_i\big|^2
= \sum_{l=1}^d\sum_{j=1}^{N_l}\mathcal E(\omega_{l,j}).
\end{equation}

It is readily seen that the proof of property 3 will be complete if we can show
\begin{equation}\label{noconcentration-F}
\lim_{R\rightarrow\infty}\lim_{i\rightarrow\infty}\int_{B_R} \lambda_i^{-2}\big|F_{B_i}\big|^2=0.
\end{equation}
Set $\widehat{t}_i=t_i+\lambda_i^2\tau_i$. Apply Uhlenbeck's gauge fixing for $A(\cdot, \widehat{t}_i)$ on $B_{\delta_0}$ for a small
$\delta_0>0$ (here $\delta_0>0$ is chosen so that for $1<p<2$, $\displaystyle\int_{B_{\delta_0}}|F_{A}|^p(x,\widehat{t}_i)$ is small,
which is possible since $\displaystyle\int_{\mathbb R^2}|F_A|^2(x,\widehat{t}_i)\le \E(0)$), we obtain a sequence of gauge transformations
$\alpha_i\in\G_{2,2}$ such that $\widehat{A}_i(x)=\alpha_i^*A(x,\widehat{t}_i)$ satisfies
\begin{equation}\label{gauge-fix1}
d^*\widehat{A}_i=0 \ {\rm{in}}\ B_{\delta_0}, \ \big\|\widehat{A}_i\big\|_{H^1(B_{\delta_0})}\le C\big\|F_{A}(\widehat{t}_i)\big\|_{L^2(B_{\delta_0})}.
\end{equation}
Set $\widehat{B}_i(x)=\lambda_i\widehat{A}_i(\lambda_i x,\widehat{t}_i)$ for $x\in B_{\delta_0\lambda_i^{-1}}$. Then we
have
\begin{equation}\label{curv-trans}
\int_{B_R} \lambda_i^{-2}\big|F_{B_i}\big|^2=\int_{B_R} \lambda_i^{-2}\big|F_{\widehat{B}_i}\big|^2
=\int_{B_R}\big|dC_i+\lambda_i [C_i,C_i]\big|^2,
\end{equation}
where $C_i(x)=\widehat{A}_i(\lambda_i x,\widehat{t}_i)$ for $x\in B_{\delta_0\lambda_i^{-1}}$.

From (\ref{e:heat1}) and (\ref{gauge-fix1}), we see that $\widehat{A}_i$ solves the elliptic equation:
\begin{equation}\label{e:heat2}
\Delta_0\widehat{A}_i(x)
=-\big(\widehat{A}_i\#d\widehat{A}_i+\widehat{A}_i\#\widehat{A}_i\#\widehat{A}_i
-\alpha_i^*\phi D_{\widehat {A}_i}\alpha_i^*\phi\big)(x,\widehat{t}_i)
-\frac{\partial (\alpha_i^*A)}{\partial t}(x,\widehat{t}_i), \ {\rm{in}}\ B_{\delta_0}.
\end{equation}
Here $\Delta_0$ denotes the standard Laplace operator on $\mathbb R^2$.
It follows from (\ref{e:heat2}) that $C_i$ solves
\begin{eqnarray}\label{e:heat3}
\Delta_0C_i(x)
&=&-\lambda_i^2\big(\widehat{A}_i\#d\widehat{A}_i+\widehat{A}_i\#\widehat{A}_i\#\widehat{A}_i
-\alpha_i^*\phi D_{\widehat {A}_i}\alpha_i^*\phi\big)(\lambda_i x,\widehat{t}_i)
-\lambda_i^2\frac{\partial (\alpha_i^*A)}{\partial t}(\lambda_i x,\widehat{t}_i)\nonumber\\
&=& I_1^i+I_2^i+I_3^i+I_4^i,
\end{eqnarray}
in  $B_{\delta_0\lambda_i^{-1}}$.
It is easy to see
\begin{equation}\label{H1-bound}
\int_{B_{\delta_0\lambda_i^{-1}}}|\nabla_0 C_i|^2=\int_{B_{\delta_0}}|\nabla_0 \widehat{A}_i|^2(\widehat{t}_i)
\le C.
\end{equation}
Using (\ref{gauge-fix1}) and (\ref{good-slice}), we can estimate
$$\big\|I_1^i\big\|_{L^{\frac43}(B_{\delta_0\lambda_i^{-1}})}
\le C\lambda_i^{\frac12}\big\|\widehat{A}_i(\widehat{t}_i)\big\|_{L^4(B_{\delta_0})}\big\|d\widehat{A}_i(\widehat{t}_i)\big\|_{L^2(B_{\delta_0})}
\le C\lambda_i^{\frac12}\big\|\widehat{A}_i(\widehat{t}_i)\big\|_{H^1(B_{\delta_0})}^2\le C\lambda_i^{\frac12},$$
$$
\big\|I_2^i\big\|_{L^{\frac43}(B_{\delta_0\lambda_i^{-1}})}
\le C\lambda_i^{\frac12}\big\|\widehat{A}_i(\widehat{t}_i)\big\|_{L^4(B_{\delta_0})}^3\le C\lambda_i^{\frac12},$$
$$\big\|I_3^i\big\|_{L^{2}(B_{\delta_0\lambda_i^{-1}})}
\le C\lambda_i\big\|D_{\widehat{A}_i}\alpha_i^*\phi(\widehat{t}_i)\big\|_{L^2(B_{\delta_0})}
=C\lambda_i\big\|D_A\phi(\widehat{t}_i)\big\|_{L^2(B_{\delta_0})}
\le C\lambda_i,$$
and
$$
\big\|I_4^i\big\|_{L^{2}(B_{\delta_0\lambda_i^{-1}})}^2
=\lambda_i^2\int_{B_{\delta_0}}\big|\frac{\partial A}{\partial t}(\widehat{t}_i)\big|^2
=\lambda_i^{-2}\int_{B_{\delta_0\lambda_i^{-1}}}\big|\frac{\partial A_i}{\partial t}(\widehat{t}_i)\big|^2=o(1)\rightarrow 0.
$$
Applying $W^{2,\frac43}$-estimate to the equation (\ref{e:heat3}),
we conclude that $C_i\in W^{2,\frac43}(B_{\frac{\delta_0}{2\lambda_i}})$,
and
\begin{eqnarray}
\big\|\nabla_0 C_i\big\|_{W^{1,\frac43}(K)}&\le&
C(K)\Big[\sum_{j=1}^4\|I_j^i\|_{L^{\frac43}(B_{\delta_0\lambda_i^{-1}})}+\|\nabla_0 C_i\|_{L^2(B_{\delta_0\lambda_i^{-1}})}\Big]
\nonumber\\
&\le& C(K)\big[1+o(1)+\lambda_i^\frac12\big]
\end{eqnarray}
for any compact subset $K\subset B_{\frac{\delta_0}{2\lambda_i}}$.
Hence we may assume, after passing to a subsequence, that
$C_i\rightarrow C$ in $H^1_{\rm{loc}}(\mathbb R^2)$. From (\ref{e:heat3}) and (\ref{H1-bound}), we see that
$$\Delta_0 C=0 \  \ {\rm{in}}\ \ \mathbb R^2, \ \int_{\mathbb R^2}|\nabla_0 C|^2<+\infty.$$
Thus $C$ must be a constant. This implies
that for any $R>0$,
\begin{equation}\label{H1-convergence}
\lim_{i\rightarrow \infty}\int_{B_R}|\nabla_0 C_i|^2=0.
\end{equation}

Observe that by (\ref{gauge-fix1})
$$\int_{B_{\delta_0\lambda_i^{-1}}}|\lambda_i^{\frac12}C_i|^4
=\int_{B_{\delta_0}}|\widehat{A}_i|^4\le C,
\ {\rm{and}}\ \int_{B_{\delta_0\lambda_i^{-1}}}\big|\nabla_0 (\lambda^{\frac12}C_i)\big|^2
=\lambda_i\int_{B_{\delta_0}}\big|\nabla_0\widehat{A}_i\big|^2\rightarrow 0.
$$
Thus we may assume that there exists a constant $\widetilde C$
such that $\lambda_i^\frac12 C_i\rightarrow \widetilde C$ in $L^4_{\rm{loc}}(\mathbb R^2)$.
Since
$$\int_{B_L}|\widetilde C|^4=\lim_{i\rightarrow \infty}\int_{B_L}|\lambda_i^{\frac12}C_i|^4\le C$$
holds for any $L>0$, we must have $\widetilde C=0$.
Hence for any $R>0$, we have
\begin{equation}\label{L4-convergence}
\lim_{i\rightarrow \infty}\int_{B_R}\big|\lambda_i[C_i, C_i]\big|^2=0.
\end{equation}
It is clear that (\ref{noconcentration-F}) follows from (\ref{curv-trans}), (\ref{H1-convergence}) and
(\ref{L4-convergence}). The  proof is now complete.
\end{proof}

For the asymptotic behavior of the global weak solution constructed by Theorem 5.1 at time infinity, we have

\begin{thm}\label{t:infinity-sing} For $(A_0,\phi_0)\in \A_{1,2}\times \S_{1,2}$, assume that  $(A, \phi)$ is the global
weak solution to the YMH  flow equation (\ref{e:heat}), under the initial condition (\ref{IVP}), constructed
by Theorem 5.1. Then the following properties hold:
\begin{enumerate}
\item There exist $t_i\uparrow +\infty$, a sequence of gauge transformations $\{s_i\}\subset\G_{2,2}$,
a set of finitely many points $\mathbf{x} := \{x_1, \cdots, x_k\}\subset \Sigma$, with $k\le \frac{\E(0)}{\epsilon_0}$,
 and  a Yang-Mills-Higgs field $(A_\infty,\phi_\infty)\in\A\times\S$ such that
 $s_i^*A(t_i)\rightarrow A_\infty$ in $H^1(\Sigma)$, $s_i^*\phi(t_i)\rightharpoonup \phi_\infty$ in $H^1(\Sigma)$,
 and $s_i^*\phi(t_i)\rightarrow \phi_\infty$
 in $H^1_{loc}(\Si\setminus\mathbf{x})$, as $i\rightarrow \infty$.
\item There exist finitely many nontrivial harmonic maps $\omega_{ij}: \mathbb S^2 \to M, 1\le i \le k, 1\le j\le l$ such that
\begin{equation}\label{energy-identity-infinity}
  \lim_{i\to \infty} \E(A(t_i), \phi(t_i)) = \E(A_\infty, \phi_\infty) + \sum_{1\le i\le k, 1\le j\le l}\E(\omega_{ij}).
\end{equation}
\item The images of $\{\omega_{ij}\}_{1\le i\le k, 1\le j\le l}$ and that of $\phi_\infty$ are connected.
\end{enumerate}
\end{thm}
\begin{proof} For simplicity, we may assume that $(A,\phi)\in C^\infty((0,+\infty), \A\times\S)$. From (\ref{e:energy-inequ}),
there exist $t_i\uparrow +\infty$ and $0<L_0<+\infty$ such that
\begin{equation}\label{energy-conv}
\lim_{i\rightarrow\infty}\E(A(t_i),\phi(t_i))=L_0,
\end{equation}
and
\begin{equation}\label{good-slice1}
\lim_{i\rightarrow\infty}\int_\Sigma \big(|\frac{\partial A}{\partial t}|^2+|\frac{\partial\phi}{\partial t}|^2\big)(x,t_i)=0.
\end{equation}
Set $(A_i(x),\phi_i(x)):=(A(x,t_i),\phi(x,t_i))$ for $x\in\Sigma$. Then $(A_i,\phi_i)\in\A\times\S$ is a sequence
of approximate YMH fields, with uniformly bounded YMH energies, such that  its
tension fields $\tau_1(A_i,\phi_i)=\frac{\partial A}{\partial t}(t_i)$ and $\tau_2(A_i,\phi_i)
=\frac{\partial\phi}{\partial t}(t_i)$. Hence it follows from (\ref{good-slice1}) that
\begin{equation}\label{vanish-tensions}
\lim_{i\rightarrow\infty}\Big[\big\|\tau_1(A_i,\phi_i)\big\|_{L^2(\Sigma)}+\big\|\tau_2(A_i,\phi_i)\big\|_{L^2(\Sigma)}\Big]=0.
\end{equation}
Applying Lemma 6.2 to $(A_i,\phi_i)$, we conclude that there exist gauge transformations $\{s_i\}\subset \G_{2,2}$,
a YMH field $(A_\infty,\phi_\infty)\in \A_{2,2}\times \S_{2,2}$ such that
$$s_i^*A_i\rightarrow A_\infty \ {\rm{in}}\ H^1(\Sigma), \ s_i^*\phi_i\rightharpoonup \phi_\infty
\ {\rm{in}}\ H^1(\Sigma).$$
The remaining parts of Theorem \ref{t:infinity-sing}, except the smoothness of $(A_\infty,\phi_\infty)$,  follow directly  from Theorem \ref{t:compact}.
While smoothness of YMH fields in $\A_{2,2}\times \S_{2,2}$, after suitable gauge transformations,
can be done by the bootstrap arguments
(see, e.g., \cite{So} Theorem 3.3).
\end{proof}

\section{Existence of global weak solutions of the YMH flow}

In this section, we indicate how to extend the local weak solution $(A,\phi)$ to the YMH flow equation (\ref{e:heat}) and
(\ref{IVP}) beyond the first singular time to a global weak solution eventually.

Note that, different form higher dimension cases, an isolated singular point of a $H^1$ connection $A$ on a Riemann
surface, in general, may not be removed even if the curvature $F$ vanishes. This can be illustrated by the following example.

Let $\D^* = \{x\in \Real^2|\ 0<|x|\le 1\}$ be the punctured disk, and $P$ is a principal $U(1)$-bundle over $\D^*$. Let $a\in \Real$ be a constant. Then $A=ad\theta$ is a well defined smooth connection on $P$. Obviously, the curvature $F$ is identically zero on the disk. However, $A$ can be extended to the whole disk by a gauge transformation if and only if the holonomy $\text{Hol(A)} = e^{2\pi a}$ equals to identity.

The next theorem shows that the limit connection $A(T)$ obtained in Theorem~\ref{t:local}
belongs to $H^1$ after a gauge transformation, which is needed for an extension of weak solutions
to (\ref{e:heat}) beyond $T$.

\begin{thm}\label{t:limit-pair} Assume $0<T_1<+\infty$ is the first singular time for the local weak solution
$(A, \phi)$ of the YMH flow equation~(\ref{e:heat}) and (\ref{IVP}), constructed by Theorem 5.1.
Then there exists a time-independent gauge transformation
$s\in \A_{2,2}$ such that $\big(\ti{A}, \ti{\phi}\big) := \big(s^*A(T_1), s^*\phi(T_1)\big) \in \A_{1,2}\times \S_{1,2}$.
\end{thm}
\begin{proof}
Let $\{t_k\}_{k=1}^\infty \subset [0, T_1)$ be a sequence approaching $T_1$. By Theorem \ref{t:local}, we have
\begin{equation}
  (A_k, \phi_k) := (A(t_k), \phi(t_k))\to (A(T_1), \phi(T_1)) \text{~in~} L^2(\Sigma).
\end{equation}
Since $(A_k, \phi_k)$ is smooth and the curvatures $F_{A_k}$ of connection $A_k$ satisfies $\norm{F_{A_k}}_{L^2(\Sigma)}$
is uniformly bounded. Then by Uhlenbeck's compactness theorem(Theorem~1.5 in~\cite{Uh}), there exists a subsequence,
still denoted by $A_k$, and a sequence of gauge transformations $\{s_k\}\subset \G_{2,2}$ such that $\widetilde{A}_k := s_k^*A_k$
converges weakly in $H^1(\Sigma)$ to a limit connection $\widetilde{A}\in \A_{1,2}$. In particular, $\widetilde{A}_k$ is bounded in $H^1(\Sigma)$.
Since
\[ \widetilde{A}_k = s_k^{-1}ds_k + s_k^{-1}A_ks_k, \]
it follows that
\[ \norm{ds_k}_{L^2(\Sigma)} \le \norm{A_k}_{L^2(\Sigma)} + \norm{\widetilde{A}_k}_{L^2(\Sigma)}\]
is bounded. Hence $s_k$ weakly converges in $H^1(\Sigma)$ to a limit gauge transformation $s\in \G_{1,2}$. Then one easily check
that $\widetilde{A}_k$ converges to $s^*A(T_1)$ weakly in $L^1(\Sigma)$. By the uniqueness of limit, we find that
$s^*A(T_1) = \widetilde{A}\in \G_{1,2}$. It is easy to check that  $s^*\phi(T_1)\in \S_{1,2}$,
since $s, \phi(T_1)\in H^1(\Sigma)\cap L^\infty(\Sigma)$. This completes the proof.
\end{proof}

\begin{rem}{\rm
It was shown by Struwe~\cite{St} Lemma 3.6 and Schlatter~\cite{Sc2} Lemma 2.4
that for Yang-Mills flow, the connection can be extended to $C^0(H^1)$ as long as the curvature dose not concentrate. This is obtained by considering the evolution equation for the curvature $F$, which turns out to be a well-behaved parabolic equation. One may suspect that for the YMH flow in dimension two, the connection $A$ may belong to $C^0(H^1)$, since the curvature does not concentrate in the subcritical dimension two. However, this may not be true since the section $\phi$ may concentrate and blow up at the singular time $T_1$ so that the equation for the curvature is not well-defined at $T_1$. Thus we have to invoke Uhlenbeck's theorem to ensure the connection $A(T_1)\in \A_{1,2}$ only after a suitable gauge transformation.}
\end{rem}

Applying Theorem~\ref{t:limit-pair}, we can prove the following  theorem on the existence of global weak solutions
to (\ref{e:heat}).

\begin{thm}\label{t:global}
Let $(A_0, \phi_0)\in \A_{1,2}\times \S_{1,2}$. There exist a global weak solution $(A, \phi)$ to the YMH flow~(\ref{e:heat})
and (\ref{IVP}) such that
\begin{itemize}
\item[i)] the energy inequality $\E(A(t),\phi(t))\le\E(A_0,\phi_0)$ holds for all $0\le t<+\infty$, and
\[ A\in C^0([0, \infty), \A_{0,2}); \phi\in C^0([0, \infty), \S_{0,2}); F_A\in L^\infty([0, \infty), L^2);  D_A\phi\in L^\infty([0,\infty), L^2).\]
\item[ii)] There exist a positive integer $L\le [\frac{\E(0)}{\alpha(M)}]$, and gauge transformations
$\{s_i\}_{i=1}^L\subset\G_{2,2}$, and $0=T_0<T_1<T_2<\cdots<T_L<+\infty$ such that for $1\le i\le L$,
$(s_i^*A,s_i^*\phi)\in C^\infty\big(\Sigma\times (T_{i-1}, T_i]\setminus \{(x_1^i,T_i),\cdots, (x_{j(i)}^i,T_i)\}\big)$ for some $j(i)\le
[\frac{\E(0)}{\alpha(M)}]$.
\end{itemize}
\end{thm}

\begin{proof}
i) By Theorem~\ref{t:local}, there exists a local weak solution $(A(t), \phi(t))$ to the YMH  flow equation~(\ref{e:heat}) on the time interval $[0, T_1)$ with initial data $(A_0, \phi_0)$. Assume $0<T_1<+\infty$ is the maximal time interval.
Then $T_1$ can be characterized by Theorem~\ref{t:local}. Let $(A_1, \phi_1)$ be the limit of
$(A(t),\phi(t))$ in $L^2(\Sigma)$ as $t\to T_1$. Then by Theorem~\ref{t:limit-pair}, we may find a gauge transformation
$s_1\in \G_{1,2}$ such that $$(\widetilde{A}_1, \widetilde{\phi}_1) := s_1^*(A_1, \phi_1) \in \A_{1,2}\times \S_{1,2}.$$
Now at  time $T_1$, we set $(\widetilde{A}_1, \widetilde{\phi}_1)$ as the initial data and apply Theorem \ref{t:local} again
to obtain a local weak solution $(\widetilde{A}(t), \widetilde{\phi}(t))$ of the YMH flow~(\ref{e:heat}) on some time interval $[T_1, T_2)$.
Since the YMH heat flow equation is invariant under time-independent gauge transformations,
$(A_1(t), \phi_1(t)) := \big(s_1^{-1}\big)^*(\widetilde{A}(t), \widetilde{\phi}(t))$, $T_1\le t<T_2$, is still a solution to the YMH flow~(\ref{e:heat}).
Since
$$(A_1, \phi_1) = (A(T_1), \phi(T_1)) = (A_1(T_1), \phi_1(T_1)),$$
we can patch $(A(t), \phi(t))$, $0\le t\le T_1$, and $(A_1(t), \phi_1(t))$, $T_1\le t<T_2$, at $T_1$ to form a new solution,
still denoted by $(A(t), \phi(t))$, to the YMH flow (\ref{e:heat}) on the time interval $[0, T_2)$ such that
\[ A(t) \in C^0([0, T_2), L^2); F(t) \in L^\infty([0, T_2), L^2); \phi(t)\in C^0([0, T_2), L^2), D_A\phi\in L^\infty([0,T_2], L^2). \]
$T_2$ can again be characterized by Theorem~\ref{t:local}. Now we can repeat the above process inductively.
Since at each singular time $T_i$,  there is an energy loss of amount at least $\alpha(M)$ by Theorem \ref{t:local},
the process stops after at most $L$ steps, for some $L\le [\frac{\E(0)}{\alpha(M)}]$.
Therefore, we obtain a global weak solution of the YMH  flow (\ref{e:heat}), which satisfies the
properties stated in Theorem \ref{t:global}.
\end{proof}

\bigskip
\noindent{\bf Acknowledgements}. Part of this work was initiated when the first author was visiting University of Kentucky in 2013, which was supported by the AMS Fan Fund China Exchange program. The first author is partially supported by NSFC No.11201387 and Natural Science Foundation of Fujian Province of China No. 2014J01023. The second author is partially supported by NSF 1522869.


\end{document}